\documentclass[11pt]{article}
\usepackage{amsmath,amsfonts,amssymb,amsthm,latexsym,mathtools}
\usepackage{graphicx,setspace}
\usepackage{xcolor}
\usepackage[title]{appendix}
\usepackage[numbers,sort&compress]{natbib}
\usepackage{enumerate,booktabs}

\usepackage[colorlinks,linkcolor=red,anchorcolor=gray,citecolor=red,urlcolor=blue]{hyperref}
\usepackage{caption}

\RequirePackage[capitalize,nameinlink]{cleveref}

\crefname{section}{section}{sections}
\crefname{subsection}{subsection}{subsections}
\Crefname{section}{Section}{Sections}
\Crefname{subsection}{Subsection}{Subsections}

\crefname{condition}{Condition}{Conditions}

\Crefname{figure}{Figure}{Figures}

\AtBeginDocument{
  \addtolength{\abovedisplayskip}{-3pt}
  \addtolength{\belowdisplayskip}{-3pt}
  \addtolength{\belowdisplayshortskip}{-3pt}
}

\crefformat{equation}{\textup{#2(#1)#3}}
\crefrangeformat{equation}{\textup{#3(#1)#4--#5(#2)#6}}
\crefmultiformat{equation}{\textup{#2(#1)#3}}{ and \textup{#2(#1)#3}}
{, \textup{#2(#1)#3}}{ and \textup{#2(#1)#3}}
\crefrangemultiformat{equation}{\textup{#3(#1)#4--#5(#2)#6}}%
{ and \textup{#3(#1)#4--#5(#2)#6}}{, \textup{#3(#1)#4--#5(#2)#6}}{ and \textup{#3(#1)#4--#5(#2)#6}}

\Crefformat{equation}{#2Equation~\textup{(#1)}#3}
\Crefrangeformat{equation}{Equations~\textup{#3(#1)#4--#5(#2)#6}}
\Crefmultiformat{equation}{Equations~\textup{#2(#1)#3}}{ and \textup{#2(#1)#3}}
{, \textup{#2(#1)#3}}{ and \textup{#2(#1)#3}}
\Crefrangemultiformat{equation}{Equations~\textup{#3(#1)#4--#5(#2)#6}}%
{ and \textup{#3(#1)#4--#5(#2)#6}}{, \textup{#3(#1)#4--#5(#2)#6}}{ and \textup{#3(#1)#4--#5(#2)#6}}

\crefdefaultlabelformat{#2\textup{#1}#3}

\newtheorem {theorem}{Theorem}[section]
\newtheorem {lemma}{{\bf Lemma}}[section]
\newtheorem {proposition}{{\bf Proposition}}[section]
\theoremstyle{remark}
\newtheorem {remark}{{\bf Remark}}[section]

\theoremstyle{definition}
\newtheorem {definition}{{\bf Definition}}[section]
\newtheorem {example}{{\bf Example}}[section]
\theoremstyle{plain} \numberwithin {equation}{section}
\newtheorem{assumption}{Assumption}

\numberwithin{assumption}{section}

\allowdisplaybreaks
\everymath{\displaystyle}
\newcommand{\R}{\mathbb R}

\newcommand{\T}{\mathbb T}
\newcommand{\dd}{\,\mathrm d}
\newcommand{\pa}{\partial}
\newcommand{\Om}{\Omega}
\newcommand{\Ga}{\Gamma}
\newcommand{\norm}[1]{\lVert#1\rVert}

\newcommand{\Id}{I}
\newcommand{\Es}{\mathcal E_s}

\newcommand{\diag}{\operatorname{diag}}

\hypersetup{pdftitle={Boundary stabilization of multi-dimensional quasilinear hyperbolic systems},pdfauthor={Haitian Yang}}
\title{\bfseries Boundary stabilization of multi-dimensional first-order quasilinear hyperbolic systems}
\author{Haitian Yang}
\date{}
\begin{document}
\maketitle
\begin{abstract}
\setlength{\emergencystretch}{2em}
We study the long-time behaviour of
multi-dimensional first-order symmetrizable quasilinear hyperbolic systems
with nonlinear feedback boundary conditions. Assuming strict interior and boundary
Lyapunov inequalities for the linearized system and non-characteristic
boundaries, we prove global existence for small initial data  and exponential $H^s$ decay, with $s>d/2+1$  an integer.  We consider both bounded smooth domains and flat periodic cylinders,
and illustrate the results with physical examples and numerical simulations.

\end{abstract}


\bigskip


\section{Introduction}\label{sec:introduction}

In this paper, we study the long-time behaviour of the following
multi-dimensional (multi-D) first-order quasilinear hyperbolic system:
\begin{equation}\label{eq:system}
 \begin{cases}
 U_t+\displaystyle\sum_{j=1}^d A_j(U)U_{x_j}=Q(U),
   &(t,x)\in(0,\infty)\times\Om,\\
 \zeta_-=\Phi(x,\zeta_+),
   &(t,x)\in(0,\infty)\times\Gamma,\\
 U(0,x)=U_0(x),&x\in\Om.
 \end{cases}
\end{equation}
Here $U=U(t,x)$ takes values in an open set $G\subset\R^n$ containing
zero. The subscripts $t$ and $x_j$ denote partial derivatives of $U$
with respect to the corresponding variables. The coefficient maps $A_j=A_j(U):G\to\R^{n\times n}$ and
$Q=Q(U):G\to\R^n$ are smooth matrices-valued functions. Hyperbolicity is understood in
the symmetrizable sense of Assumption~\ref{ass:sym}.
We take the spatial domain $\Om$ to be either a bounded $C^\infty$ domain in $\R^d$
or the periodic cylinder $(0,1)\times\T^{d-1}$, where $d\ge2$. We denote its boundary by $\Gamma.$
At each boundary point $x \in \Ga$, $\zeta_+=\zeta_+(t,x) \in \R^{n_+}$ and $\zeta_-=\zeta_-(t,x) \in \R^{n_-}$ represent outgoing
and incoming information, respectively. They are linear combinations
of $U$ at the boundary.  $\Phi(x,\zeta_+)$ is a smooth nonlinear boundary law.
We will explain these variables
and the boundary condition soon.

Our purpose is to study the following question. Suppose that
$0\in\R^n$ is a steady state (also called an equilibrium) of \eqref{eq:system}, i.e., $Q(0)=0$ and $\Phi(x,0)=0$,
and that the linearization of \eqref{eq:system} around $U=0$
admits a strict quadratic $L^2$ Lyapunov functional under
Assumptions~\ref{ass:sym} and~\ref{ass:boundary}. Do sufficiently small solutions of
\eqref{eq:system} exist for all time and decay exponentially in
the higher Sobolev norm $H^s$? In terms of boundary control, the above question can be restated as follows: if
the linearized boundary law determined by $\Phi_{\zeta_+}(x,0)$
stabilizes the linearized system, does the nonlinear law
$\zeta_-=\Phi(x,\zeta_+)$ stabilize \eqref{eq:system}? Here and below, $\Phi_{\zeta_+}$ denotes the Jacobian of $\Phi$ with respect to $\zeta_+$.

A further central hypothesis is that the boundary is
non-characteristic at equilibrium; see
\cite[Definition~4.2]{BGS2007}. Set
$A_j^0:=A_j(0)$ and $Q_0:=Q_U(0)$, where $Q_U$ denotes the Jacobian of $Q(U)$ with respect to $U$. For the outward unit normal vector
$\nu(x)$ at $x\in\Ga$, we write
\begin{equation*}
 A_\nu(x,U):=\sum_{j=1}^d\nu_j(x)A_j(U),\qquad
 A_\nu^0(x):=A_\nu(x,0).
\end{equation*}
\begin{assumption}[Non-characteristic boundary]\label{ass:noncharacteristic}
The matrix $A_\nu^0(x)$ is invertible for every $x\in\Ga$.
\end{assumption}
On the periodic cylinder $\Om=(0,1)\times\T^{d-1}$, this assumption is precisely the
invertibility of $A_1(0)$, since the boundary $\Ga$ consists of only two faces $\Ga_0:=\{0\}\times\T^{d-1}$ and $\Ga_1:=\{1\}\times\T^{d-1}$.

We now define the boundary variables. Symmetrizability (\ref{eq:symmetrizer}) in
Assumption~\ref{ass:sym} implies that $A_\nu^0(x)$ is similar to a symmetric matrix. Non-characteristicity separates the negative and positive spectra of $A_\nu^0(x)$.
Hence, near each boundary point $x$, we can choose a smooth invertible
matrix $\Pi=\Pi(x)$ such that
\begin{equation*}
 \Pi^{-1}A_\nu^0\Pi=
 \begin{pmatrix}\Lambda_-&0\\0&\Lambda_+\end{pmatrix},
\end{equation*}
where $\Lambda_-=\Lambda_-(x)\in\R^{n_-\times n_-}$ and
$\Lambda_+=\Lambda_+(x)\in\R^{n_+\times n_+}$ have strictly negative and positive
eigenvalues of $A_\nu^0(x)$, respectively, and $n_-+n_+=n$. We write
\begin{equation*}
 \zeta=\zeta(t,x)=\begin{pmatrix}\zeta_-\\\zeta_+\end{pmatrix}:=\Pi^{-1}(x)U(t,x).
\end{equation*}
The row blocks of $\Pi^{-1}$ are denoted by $L_\pm$, so that
$\zeta_\pm=L_\pm U$. From a control perspective, the outgoing
variable $\zeta_+$ carries information leaving the domain and
can be measured by sensors on the boundary. The incoming
variable $\zeta_-$ carries information entering the domain and
can be modified by actuators through the boundary input. The
feedback $\zeta_-=\Phi(x,\zeta_+)$ thus makes the incoming
variable depend nonlinearly on the outgoing variable.



Boundary control and stabilization of one-dimensional (1-D) first-order
hyperbolic systems have been extensively studied. The main approaches
include characteristic methods \cite{Li2010}, backstepping methods
\cite{HDVK2016}, and Lyapunov methods \cite{CBDN2007,CBDN2008,BC2016}.
Constructive feedback design is well developed for linear systems,
and nonlinear extensions are also available. In particular,
backstepping yields local exponential stabilization of quasilinear
systems by full-state boundary feedback \cite{CVKB2013,HVDK2019}.
Under suitable strict Lyapunov and boundary dissipation conditions,
one obtains stability in $H^2$ spaces \cite{BC2016,WWY2020,HS2021,GLTW2012}. Moreover,
stability results in the $C^1$ norm can also be established \cite{CB2015,CN2015,PWB2008}.

In recent years, boundary stabilization of multi-D first-order
hyperbolic systems has attracted growing interest
\cite{HTControl2024,YY2024,YY2025,BaiHayat2026}.
However, these works address only linear systems. Extending these results to quasilinear systems requires control
of the full $H^s$ norm, with $s>d/2+1$, to close the nonlinear estimates. However, the boundary condition can
be differentiated in time and in tangential directions, but not in the
normal direction. Consequently, the normal derivatives of $U$ do not directly
inherit a dissipative boundary condition, which obstructs direct energy
estimates for them. 

In 1-D, the equation provides a way to overcome this difficulty.
If $A_1(0)$ is invertible, then, for small $U$,
\begin{equation}\label{eq:one-dimensional-recovery}
 U_{x_1}=A_1^{-1}(U)\bigl(Q(U)-U_t\bigr).
\end{equation}
This relation between time and spatial derivatives is used in
\cite{CBDN2007}. Recursive use of
\eqref{eq:one-dimensional-recovery} gives, along sufficiently
small solutions,
\begin{equation}\label{eq:intro-time-equivalence}
 \|U(t,\cdot)\|_{H^s(\Om)}^2
 \asymp\sum_{k=0}^{s}
 \|\partial_t^kU(t,\cdot)\|_{L^2(\Om)}^2.
\end{equation}
Here $\asymp$ means that the two sides are equivalent up to a constant factor.
Thus an energy involving only time derivatives can control the
full spatial regularity. Furthermore, note that the invertibility of $A_1(0)$ is precisely
the 1-D non-characteristic boundary condition. This observation
motivates the use of non-characteristicity to recover spatial
regularity in higher dimensions. In the following, we consider
two different types of domains.

On a bounded smooth domain, the non-characteristic boundary
condition implies ellipticity of the spatial operator at $U=0$.
Together with the dissipative boundary condition, this yields an elliptic
boundary estimate and hence the equivalence
\eqref{eq:intro-time-equivalence} for sufficiently small solutions.
The ellipticity imposed by this geometry excludes some important hyperbolic
systems, including the compressible Euler equations. Therefore, we also consider
the periodic cylinder $\Om=(0,1)\times\T^{d-1}$, whose
boundary consists of two faces. Let $x'=(x_2,\ldots,x_d)\in\T^{d-1}$ denote
the tangential spatial variables. We estimate time and tangential
derivatives and recover normal derivatives recursively from the equation.
For sufficiently small solutions, this gives
\begin{equation}\label{eq:intro-tangential-equivalence}
 \|U(t,\cdot)\|_{H^s(\Om)}^2
 \asymp\sum_{k+|\alpha|\le s}
 \|\partial_t^k\partial_{x'}^\alpha U(t,\cdot)\|_{L^2(\Om)}^2.
\end{equation}
The equivalences \eqref{eq:intro-time-equivalence} and
\eqref{eq:intro-tangential-equivalence} allow us to close the
nonlinear energy estimates and prove global existence and
exponential $H^s$ stability for small compatible initial data.

We point out an exception to the non-characteristic hypothesis:
multi-D quasilinear Maxwell systems admit exponential
decay under absorbing boundary feedback \cite{PS2020}. For these systems, the normal matrix $A_\nu^0(x)$ is not
invertible, so Assumption~\ref{ass:noncharacteristic} is not satisfied.
The additional div--curl structure, together with the boundary
condition, nevertheless provides elliptic estimates that recover
spatial derivatives from time derivatives. 

Finally, we give a counterexample showing that, unlike in the 1-D case,
our assumptions do not ensure an exponential $C^1$ decay estimate
in terms of the initial $C^1$ norm,
owing to derivative loss in multi-D wave propagation;
see Section~\ref{sec:c1-counterexample}.

The paper is organized as follows.
Section~\ref{sec:setting} gives the assumptions and  main theorems.
Section~\ref{sec:smooth-proof} proves the smooth-domain result using
elliptic recovery. Section~\ref{sec:cylinder-proof} proves the
cylinder result using tangential estimates and normal recovery.
Section~\ref{sec:c1-counterexample}
gives a counterexample to such a $C^1$ decay estimate.
Appendix~\ref{app:nonlinear} collects the nonlinear estimates used in Sections~\ref{sec:smooth-proof} and~\ref{sec:cylinder-proof}.

\paragraph{Acknowledgements.}
This work forms part of my doctoral dissertation.
I am deeply grateful to my PhD advisor, Prof.~Wen-An Yong,
for introducing me to the study of multi-D hyperbolic systems. I thank Dr.~Yubo Bai for discussions that helped me recognize
the importance of elliptic operators.
This work was initially concerned with semilinear systems.
I thank Prof.~Christophe Prieur and Prof.~Jean-Michel Coron
for suggesting the extension to quasilinear systems.

\section{Assumptions and main results}\label{sec:setting}

We state the common dissipation assumptions and the local existence
theorem, using the boundary variables defined in
Section~\ref{sec:introduction}. We then give the two stabilization
results for the different domain geometries. We illustrate Theorems~\ref{thm:smooth} and~\ref{thm:cylinder}
with two examples and numerical simulations.

\subsection{Assumptions and local well-posedness}
We assume that system~\eqref{eq:system} is symmetrizable
and that its linearization satisfies interior and boundary
dissipation conditions.
\begin{assumption}[Symmetrizability and internal dissipation]\label{ass:sym}
There are a smooth symmetric matrix $M=M(x,U)$
and constants $C\ge1$ and $\lambda_0>0$ such that, for $U\in G$,
\begin{equation}\label{eq:symmetrizer}
 C^{-1}\Id\le M(x,U)\le C\Id,\qquad
 M(x,U)A_j(U)=A_j^\top(U)M(x,U),\quad j=1,\ldots,d.
\end{equation}
At equilibrium, set $M_0(x)=M(x,0)$, $A_j^0=A_j(0)$,
and $Q_0=Q_U(0)$. We assume that
\begin{equation}\label{eq:interior}
 \sum_{j=1}^d(M_0A_j^0)_{x_j}+M_0Q_0+Q_0^\top M_0
 \le-2\lambda_0M_0.
\end{equation}
\end{assumption}
\begin{assumption}[Boundary dissipation]\label{ass:boundary}
There is $\delta>0$ such that, for every $x\in\Ga$ and every
$U\in\R^n$ satisfying  $\zeta_-=K(x)\zeta_+,$ one has
\begin{equation*}
 U^\top M_0(x)A_\nu^0(x)U\ge\delta|\zeta_+|^2,
\end{equation*}
where $\zeta=\Pi^{-1}U=(\zeta_-^\top,\zeta_+^\top)^\top$ and $K(x)=\Phi_{\zeta_+}(x,0)$.
\end{assumption}
Once the symmetrizer $M_0$ is fixed, Assumption~\ref{ass:boundary}
holds if the entries of $K(x)$ are uniformly sufficiently small;
see \cite{YY2024}. Physically, the incoming energy flux is smaller
than the outgoing energy flux. Note that
Assumptions~\ref{ass:sym} and~\ref{ass:boundary} imply that
the linearized system
\begin{equation*}
 \begin{cases}
 U_t+\displaystyle\sum_{j=1}^d A_j(0)U_{x_j}=Q_U(0)U,
   &(t,x)\in(0,\infty)\times\Om,\\
 \zeta_-=K(x)\zeta_+,
   &(t,x)\in(0,\infty)\times\Ga,\\
   U(0,x)=U_0(x),&x\in\Om,
 \end{cases}
\end{equation*}
is exponentially stable in $L^2(\Om)$, with Lyapunov functional
\begin{equation*}
 \mathcal V(U)=\int_\Om U^\top M(x,0)U\dd x,
 \qquad
 \frac{\dd}{\dd t}\mathcal V(U(t))\le-2\lambda_0\mathcal V(U(t)).
\end{equation*}

 The following lemma estimates the boundary flux under nonhomogeneous
boundary conditions, as needed for higher-order energy estimates.
\begin{lemma}[Nonhomogeneous boundary flux]\label{lem:inhom-flux}
Under Assumptions~\ref{ass:noncharacteristic}, \ref{ass:sym},
and~\ref{ass:boundary}, there exist $R,c,C>0$ such that, for
$x\in\Ga$, $|u|\le R$, and any $U\in\R^n$ satisfying
$\zeta_-=K(x)\zeta_++g$, one has
\begin{equation*}
 U^\top M(x,u)A_\nu(x,u)U\ge c|\zeta_+|^2-C|g|^2,
\end{equation*}
where $\zeta=\Pi^{-1}U=(\zeta_-^\top,\zeta_+^\top)^\top$
and $g\in\R^{n_-}$.
\end{lemma}
\begin{proof}
Set
\begin{equation*}
 U=W+Z,\qquad
 W=\Pi\binom{K\zeta_+}{\zeta_+},\qquad
 Z=\Pi\binom{g}{0}.
\end{equation*}
By smoothness and Assumption~\ref{ass:boundary},
\begin{align*}
 U^\top M(x,u)A_\nu(x,u)U
 &=W^\top M_0A_\nu^0W
   +2W^\top M_0A_\nu^0Z
   +Z^\top M_0A_\nu^0Z
   +\mathcal O(|u||U|^2)\\
 &\ge\delta|\zeta_+|^2-C|\zeta_+||g|
   -C|g|^2-C|u||U|^2.
\end{align*}
Using $|U|^2\le C(|\zeta_+|^2+|g|^2)$ and Young's inequality,
we obtain, for $|u|\le1$,
\begin{equation*}
 U^\top M(x,u)A_\nu(x,u)U
 \ge\left(\frac\delta2-C|u|\right)|\zeta_+|^2-C|g|^2.
\end{equation*}
Choose $R\le1$ so that $CR\le\delta/4$.
The estimate follows with $c=\delta/4$.
\end{proof}

For an integer $s>d/2+1$, compatibility of order $s-1$ means
that the initial time derivatives, computed recursively from
\eqref{eq:system}, satisfy the time-differentiated boundary condition
through order $s-1$; see \cite[Section~11.1.2]{BGS2007}.
We use the solution space
\begin{equation*}
 \mathcal X^s([0,T])
 :=\bigcap_{k=0}^s C^k([0,T];H^{s-k}(\Om;\R^n)).
\end{equation*}
We use a small-data local theory, since the structural assumptions
hold near the zero equilibrium.
Following \cite[Appendix~A, Theorem~A2]{Schochet1986} and
\cite[Section~1.1]{Trakhinin2006}, we have the following local
existence and continuation theorem.

\begin{theorem}[Local existence and continuation]\label{thm:local}
Let $\Om$ be either of the domains introduced in
Section~\ref{sec:introduction}, and let $s>d/2+1$ be an integer.
Under Assumptions~\ref{ass:noncharacteristic} and~\ref{ass:boundary},
together with \eqref{eq:symmetrizer} in Assumption~\ref{ass:sym},
there exists $\delta_0>0$ such that, for
every $U_0\in H^s(\Om;\R^n)$ satisfying
\begin{equation*}
 \norm{U_0}_{H^s(\Om)}\le\delta_0
\end{equation*}
and the compatibility conditions of order $s-1$, problem
\eqref{eq:system} has a unique maximal classical solution $U$ on $[0,T_{\max})$,
with $T_{\max}\in(0,\infty]$, such that
$U\in\mathcal X^s([0,T])$ for every $0<T<T_{\max}$.
Moreover, if
\begin{equation*}
 \norm{U(t,\cdot)}_{H^s(\Om)}\le\delta_0
 \qquad\text{for all }t\in[0,T_{\max}),
\end{equation*}
then $T_{\max}=\infty$.
\end{theorem}
\subsection{A bounded smooth Euclidean domain}

Let us state the first main result, which uses ellipticity of the first-order operator. 
\begin{theorem}\label{thm:smooth}
Let $\Om\subset\R^d$, $d\ge2$, be a bounded $C^\infty$ domain.
Assume Assumptions~\ref{ass:noncharacteristic}, \ref{ass:sym},
and~\ref{ass:boundary}.
For each integer $s>d/2+1$ and each $0<\lambda<\lambda_0$,
there exist $\varepsilon,C>0$ such that every
$U_0\in H^s(\Om;\R^n)$ satisfying the compatibility conditions
of order $s-1$ and $\norm{U_0}_{H^s(\Om)}\le\varepsilon$ generates
a unique global solution of \eqref{eq:system}
in $\mathcal X^s([0,T])$ for every finite $T$, with
\begin{equation*}
 \norm{U(t,\cdot)}_{H^s(\Om)}\le C e^{-\lambda t}\norm{U_0}_{H^s(\Om)}
 \qquad(t\ge0).
\end{equation*}
\end{theorem}

\begin{example}[Planar elasticity with nonlinear volumetric response]
\label{ex:nonlinear-elasticity}
The velocity--dilatation--rotation formulation of elastic waves is
used in \cite{THM2016}. The following choice of constitutive law,
bulk damping, elastic foundation, and boundary feedback gives a
quasilinear example of Theorem~\ref{thm:smooth}.
Let $d=2$, $c_p>c_s>0$, and $\beta,\gamma>0$.
For a displacement $w(t,x)\in\R^2$, set
$e=\operatorname{div}w$,
$\omega=w_{2,x_1}-w_{1,x_2}$, $v=w_t+\gamma w$, and
$U=(e,\omega,v_1,v_2)^\top\in G$, where $G\subset\R^4$
is a sufficiently small open neighborhood of zero.
The closed first-order system is
\begin{equation*}
 \begin{cases}
 e_t-\operatorname{div}v=-\gamma e,\\
 \omega_t-(v_{2,x_1}-v_{1,x_2})=-\gamma\omega,\\
 v_t-\nabla h(e)-c_s^2J\nabla\omega=-\gamma v,
 \end{cases}
 \qquad
 h(e)=c_p^2e+\beta e^3,\qquad
 J=\begin{pmatrix}0&-1\\1&0\end{pmatrix}.
\end{equation*}
With $a(e)=h'(e)$, its matrices and symmetrizer are
\begin{equation*}
 A_1(U)=\begin{pmatrix}
 0&0&-1&0\\0&0&0&-1\\-a(e)&0&0&0\\0&-c_s^2&0&0
 \end{pmatrix},\qquad
 A_2(U)=\begin{pmatrix}
 0&0&0&-1\\0&0&1&0\\0&c_s^2&0&0\\-a(e)&0&0&0
 \end{pmatrix},
\end{equation*}
with $M(U)=\operatorname{diag}(a(e),c_s^2,1,1)$ and $Q(U)=-\gamma U$.
Both $M(U)A_j(U)$ are symmetric. The eigenvalues of the frozen
symbol are $\pm c_p|\xi|$ and $\pm c_s|\xi|$, so every smooth
boundary is non-characteristic. Moreover, the left side of
\eqref{eq:interior} equals $-2\gamma M_0$.
At the boundary, let $\tau=J\nu$ with $\nu$ the unit outward normal vector and prescribe
\begin{equation*}
 v\cdot\nu=-c_pe,\qquad v\cdot\tau=-c_s\omega.
\end{equation*}
These conditions set the incoming variables to zero,
so $\Phi=0$. Consequently, we have
\begin{equation*}
 U^\top M_0A_\nu^0U
 =2c_p^3e^2+2c_s^3\omega^2
 \ge c_s\min\{1,c_s^2\}|U|^2.
\end{equation*}
All assumptions therefore hold with $\lambda_0=\gamma$. Theorem~\ref{thm:smooth} then yields a unique global solution
with exponential $H^s$ decay for this system with sufficiently small compatible initial data, where $s>2$ is an integer.

\end{example}

\subsection{A flat periodic cylinder}

The second result uses flatness and periodicity to replace full
spatial ellipticity by invertibility in the normal direction.
Write $x=(x_1,x')$ and set
\begin{equation*}
 \Om=(0,1)\times\T^{d-1},\qquad
 \Ga_i=\{i\}\times\T^{d-1},\quad i=0,1.
\end{equation*}
The outward normals are $\nu_0=-e_1$ and $\nu_1=e_1$, so
$A_{\nu_0}^0=-A_1^0$ and $A_{\nu_1}^0=A_1^0$.
Consequently, the positive-speed modes of $A_1^0$ are incoming
at $\Ga_0$ and outgoing at $\Ga_1$; the negative-speed modes have
the opposite roles. On each face, the boundary condition is
\begin{equation}\label{eq:cylinder-graph}
 \zeta_{i,-}=\Phi_i(x',\zeta_{i,+}).
\end{equation}
Note that $\Phi_i(x',0)=0$ and set $K_i(x')=\Phi_{i,\zeta_+}(x',0)$.

\begin{theorem}\label{thm:cylinder}
Let $\Om=(0,1)\times\T^{d-1}$, $d\ge2$, and assume
Assumptions~\ref{ass:noncharacteristic}, \ref{ass:sym},
and~\ref{ass:boundary}, with the boundary graphs
\eqref{eq:cylinder-graph} on both faces.
For each integer $s>d/2+1$ and each $0<\lambda<\lambda_0$,
there exist $\varepsilon,C>0$ such that every
$U_0\in H^s(\Om;\R^n)$ satisfying compatibility of order $s-1$
and $\norm{U_0}_{H^s(\Om)}\le\varepsilon$ generates
a unique global solution of \eqref{eq:system}
in $\mathcal X^s([0,T])$ for every finite $T$, with
\begin{equation*}
 \norm{U(t,\cdot)}_{H^s(\Om)}\le C e^{-\lambda t}\norm{U_0}_{H^s(\Om)}
 \qquad(t\ge0).
\end{equation*}
\end{theorem}

\begin{example}[Subsonic gas flow with damping]\label{ex:damped-euler}
Consider a gas in a thin annular duct of length $L$, constant mean
radius $R_a$, and constant gap width $h\ll R_a$. Integrating across
the gap and neglecting radial variations gives a two-dimensional (2-D)
axial--circumferential Euler model; see
\cite[Section~2, equations~(2.1a--d)]{lit-c2d5f55c2bba}.
We take impermeable slip walls at the inner and outer radii and
specialize to isentropic flow. The pressure law $p=p(\rho)$ is
smooth for $\rho>0$, with $p'(\rho)>0$; for example,
$p(\rho)=a\rho^\gamma$ with $a>0$ and $\gamma>1$.

Let $X$ be the physical axial coordinate and $\theta$ the angular
coordinate. With $x_1=X/L$, $x_2=\theta/(2\pi)$, and a common velocity
scale $V_*$, scale time by $L/V_*$ and put
$\beta=L/(2\pi R_a)>0$. In these dimensionless variables,
$\Om=(0,1)\times\T$. Here $x_2$ is the tangential spatial variable $x'$ introduced above. We add linear velocity damping and a constant
axial drive, obtaining
\begin{equation*}
 \begin{cases}
 \rho_t+u\rho_{x_1}+\beta v\rho_{x_2}
       +\rho(u_{x_1}+\beta v_{x_2})=0,\\
 u_t+u u_{x_1}+\beta v u_{x_2}
       +\dfrac{p'(\rho)}{\rho}\rho_{x_1}=-\kappa(u-\bar u),\\
 v_t+u v_{x_1}+\beta v v_{x_2}
       +\beta\dfrac{p'(\rho)}{\rho}\rho_{x_2}=-\kappa v,
 \end{cases}
\end{equation*}
where $\kappa>0$. The axial source combines drag $-\kappa u$
with the driving acceleration $\kappa\bar u$, so that
$(\bar\rho,\bar u,0)$ is a steady state. The damping and drive
are additional modelling assumptions in this example.
Fix $\bar\rho>0$ and a subsonic velocity
$0<\bar u<\bar c:=\sqrt{p'(\bar\rho)}$.

For $U=(\rho-\bar\rho,u-\bar u,v)^\top$, the state space $G$
is a sufficiently small open neighborhood of zero contained in
$\{U\in\R^3:\bar\rho+U_1>0\}$. The system has
\begin{equation*}
 \begin{gathered}
 A_1(U)=\begin{pmatrix}
 u&\rho&0\\p'(\rho)/\rho&u&0\\0&0&u
 \end{pmatrix},\qquad
 A_2(U)=\beta\begin{pmatrix}
 v&0&\rho\\0&v&0\\p'(\rho)/\rho&0&v
 \end{pmatrix},\\
 Q(U)=(0,-\kappa U_2,-\kappa U_3)^\top,\qquad
 M(x,U)=(K_0-x_1)\diag\bigl(p'(\rho)/\rho^2,1,1\bigr).
 \end{gathered}
\end{equation*}
Let $\bar U=(\bar\rho,\bar u,0)^\top$ denote the steady state,
which corresponds to $U=0$ in the perturbation variables.
Since $\det A_1(0)=\bar u(\bar u^2-\bar c^2)\ne0$,
Assumption~\ref{ass:noncharacteristic} holds.
Following \cite{YY2024}, the matrix $M(x,U)$ satisfies
Assumption~\ref{ass:sym} for sufficiently large $K_0$.
We choose a strictly dissipative boundary feedback satisfying
Assumption~\ref{ass:boundary}, for example $\zeta_{i,-}=0$ on both faces.
Theorem~\ref{thm:cylinder} then yields a unique global solution
with exponential $H^s$ decay for sufficiently small compatible
perturbations of $\bar U$, where $s>2$ is an integer.
\end{example}

\subsection{Numerical illustrations}

We illustrate the two examples with numerical simulations. Define
\begin{equation*}
 \psi(z)=\begin{cases}
 \exp\bigl(1-(1-z^2)^{-2}\bigr),& |z|<1,\\
 0,& |z|\ge1.
 \end{cases}
\end{equation*}
For Example~\ref{ex:nonlinear-elasticity}, take the unit disk
$\Om=B_1(0)$, $c_p=3/2$, $c_s=1$, $\beta=1$, and $\gamma=0.3$,
with $e_0=\omega_0=\psi(4|x|/3)$ and $v_0=0$.
We use the boundary feedback specified in the example.
Rotational symmetry reduces the computation to the radial variable.
We use a Chebyshev collocation method of degree $191$ on the signed
radial interval $[-1,1]$, with even scalar fields and odd radial
and azimuthal velocity components. The full Cartesian $H^3$ norm
is evaluated by radial Gaussian quadrature and $24$ equally spaced
angular points. Time integration uses the classical fourth-order
Runge--Kutta method with $\Delta t=5\times10^{-4}$.

For Example~\ref{ex:damped-euler}, take $p(\rho)=\rho^2/2$,
$\bar\rho=1$, $\bar u=1/2$, $\kappa=1$, and $\beta=1/2$.
Choose $U_0=0.1\psi(\tfrac52(x_1-\tfrac12))\cos(2\pi x_2)(1,0,0)^\top$.
This perturbation is supported in $1/10<x_1<9/10$ and is periodic in $x_2$.
We impose $u+\rho=3/2$ and $v=0$ at $x_1=0$, and
$u-\rho=-1/2$ at $x_1=1$.
One can verify that these boundary conditions satisfy the dissipation
condition in Assumption~\ref{ass:boundary}.
The nonlinear equations are discretized at $385$ Chebyshev--Lobatto
points in $x_1$ and $56$ equally spaced Fourier points in $x_2$,
using the same time integrator with $\Delta t=5\times10^{-5}$.
The boundary condition is enforced at every Runge--Kutta stage.
Figure~\ref{fig:h3-decay} plots the natural logarithm of the
$H^3$ norm and illustrates its decay in both examples.

\begin{figure}[htbp]
 \centering
 \includegraphics[width=\textwidth]{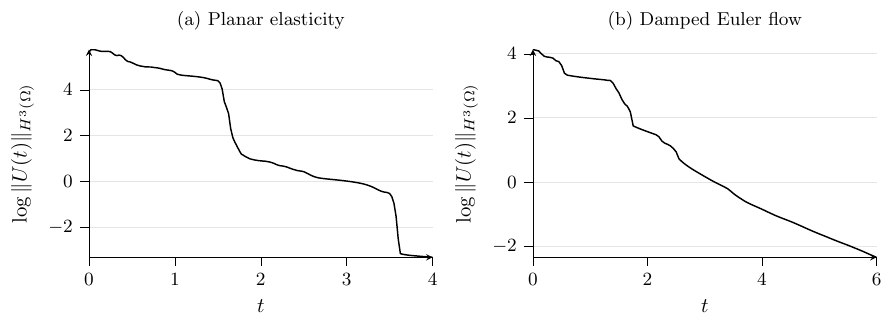}
 \caption{Decay of the  $H^3$ norm in
 Examples~\ref{ex:nonlinear-elasticity} and~\ref{ex:damped-euler}.}
 \label{fig:h3-decay}
\end{figure}

\section{Proof of Theorem~\ref{thm:smooth}}\label{sec:smooth-proof}

\subsection{Elliptic preliminaries}

We first recall the definition of the first-order elliptic operator and prove the ellipticity of the frozen principal operator. It will be used to recover all spatial derivatives needed for a Lyapunov energy
 from time derivatives.
\begin{definition}[Elliptic operator]\label{def:elliptic-operator}
Let $P=\sum_{j=1}^d C_j(x)\pa_{x_j}+C_0(x)$ be a first-order
operator with smooth $n\times n$ matrix coefficients on
$\overline\Om$.
Its principal symbol is $p(x,\xi)=\sum_{j=1}^d C_j(x)\xi_j$ with $\xi:=(\xi_1,\ldots,\xi_d) \in \R^d.$ The operator $P$ is
elliptic if $p(x,\xi)$ is invertible for every
$x\in\overline\Om$ and every real $\xi\in \R^d \backslash \{0\}$.
\end{definition}

\begin{proposition}\label{prop:geometry}
Under the hypotheses of Theorem~\ref{thm:smooth}, the spatial
operator $\sum_{j=1}^d A_j^0\pa_{x_j}$ is elliptic, with $A_j^0:=A_j(0) \in \R^{n\times n}.$
Moreover, $n$ is even, and the matrix $A_\nu^0(x):=\sum_{j=1}^d A_j^0\nu_j(x)$
has $n/2$
eigenvalues of each sign, counted with multiplicity for every point $x\in\Ga$.
\end{proposition}

\begin{proof}
The non-characteristic boundary assumption implies that
$\sum_{j=1}^d\nu_j(x)A_j^0$ is invertible for every $x\in\Ga$.
Since $\Om$ is bounded with smooth boundary, its outward unit
normal map $\nu:\Ga\to S^{d-1}$ is surjective
\cite{BallZarnescu2017}, where $S^{d-1}=\{\omega\in\R^d:|\omega|=1\}$ is the unit sphere.
By homogeneity, $\sum_{j=1}^d\xi_jA_j^0$ is therefore invertible
for every $\xi\in\R^d\setminus\{0\}$, which proves ellipticity.

Fix $x\in\Ga$. For every $\omega\in S^{d-1}$, the matrix
\[
 B(\omega):=M_0^{1/2}(x)
 \left(\sum_{j=1}^d\omega_jA_j^0\right)M_0^{-1/2}(x)
\]
is symmetric by Assumption~\ref{ass:sym} and invertible by
ellipticity. Since $S^{d-1}$ is connected and no eigenvalue
can pass through zero, the numbers of positive and negative
eigenvalues are constant in $\omega$.
The identity $B(-\omega)=-B(\omega)$ interchanges these two
numbers, so they must agree. Thus $n$ is even and each number
equals $n/2$.
Taking $\omega=\nu(x)$ and using the similarity of
$B(\nu(x))$ and $A_\nu^0(x)$ proves the assertion.
\end{proof}

Define the frozen spatial operator and its linear boundary map by
\begin{equation*}
 P_0v=\sum_{j=1}^d A_j^0v_{x_j}-Q_0v,\qquad
 B_0v=L_-v-K(x)L_+v
\end{equation*}
where $L_-$ and $L_+$ are the row blocks of $\Pi^{-1}$
corresponding to the incoming and outgoing characteristic
variables, respectively.
 We now recall the Shapiro--Lopatinskii condition, which is a necessary and sufficient condition for the elliptic boundary estimate.
\begin{definition}[Shapiro--Lopatinskii condition]\label{def:complementing}
For the first-order operator $P_0$, suppose that
$\operatorname{rank}B_0(x)=n/2$ for every $x\in\Ga$.
Write
\begin{equation*}
 p(x,\eta)=\sum_{j=1}^d A_j^0\eta_j,
 \qquad \eta=(\eta_1,\ldots,\eta_d)\in\R^d.
\end{equation*}
At a boundary point $x$, consider the differential equation
\begin{equation}\label{eq:elliptic-normal-model}
 -p(x,\nu(x))\frac{dw}{dr}+ip(x,\eta)w=0\quad(r>0),
 \qquad w(r)\longrightarrow0\quad(r\longrightarrow+\infty),
\end{equation}
with $i$ the imaginary unit.
Suppose moreover that, for every $x\in\Ga$ and every nonzero
$\eta\in\R^d$ satisfying $\eta\cdot\nu(x)=0$, the only solution of
\eqref{eq:elliptic-normal-model} satisfying $B_0(x)w(0)=0$ is $w\equiv0$.
Then we say that $B_0$ satisfies the Shapiro--Lopatinskii condition
for $P_0$; see \cite[Chapter~5, Section~4.5 and Section~6.7(b)--(c)]{CP1982}.
This condition is also called the elliptic complementing condition \cite{ADN1964}.
\end{definition}

For an elliptic operator equipped with a boundary condition satisfying
the Shapiro--Lopatinskii condition, we have the following elliptic
estimate; see \cite[Chapter~5, Theorem~4.4(ii) and Section~6.7(e)]{CP1982}. This elliptic estimate is a key ingredient in the proof of Theorem~\ref{thm:smooth}.
\begin{lemma}[Elliptic boundary estimate]\label{lem:elliptic-estimate}
Suppose that $P_0$ is a smooth first-order elliptic operator on
$\overline\Om$, where $\Om$ is bounded and smooth, and that the
smooth surjective boundary operator $B_0$ 
satisfies Definition~\ref{def:complementing}.
For every integer $m\ge1$, there is $C_m>0$ such that every
$v\in H^m(\Om;\R^n)$ solving the boundary value problem
\begin{equation*}
 \begin{cases}
 P_0v=f&\text{in }\Om,\\
 B_0v=g&\text{on }\Ga
 \end{cases}
\end{equation*}
satisfies
\begin{equation*}
 \norm v_{H^m(\Om)}\le C_m\bigl(
 \norm f_{H^{m-1}(\Om)}+
 \norm g_{H^{m-1/2}(\Ga)}+\norm v_{L^2(\Om)}\bigr).
\end{equation*}
The constant $C_m$ is independent of $v$.
\end{lemma}

We now prove that $B_0$ satisfies the Shapiro--Lopatinskii condition. 
\begin{proposition}\label{prop:elliptic}
Under the hypotheses of Theorem~\ref{thm:smooth}, the boundary
operator $B_0=L_--KL_+$ satisfies the elliptic complementing
condition for $P_0$.
\end{proposition}

\begin{proof}
Since $B_0\Pi=(I_{n_-},-K)$ and $\Pi$ is invertible,
Proposition~\ref{prop:geometry} gives
$\operatorname{rank}B_0=n_-=n/2$.
Fix $x\in\Ga$ and a nonzero vector $\eta$ satisfying $\eta \cdot \nu(x)=0$.
Let $w$ solve \eqref{eq:elliptic-normal-model} with $B_0(x)w(0)=0$.
Multiplying the equation in \eqref{eq:elliptic-normal-model} by $w^*M_0(x)$, taking real parts,
and integrating over $(0,\infty)$, we have
\begin{equation*}
 0=-\int_0^\infty\operatorname{Re}\left(
 w^*M_0(x)p(x,\nu(x))\frac{dw}{dr}\right)\dd r
 =\frac12 w(0)^*M_0(x)p(x,\nu(x))w(0).
\end{equation*}
Assumption~\ref{ass:boundary}, applied to the real and imaginary
parts of $w(0)$, yields $w(0)=0$.
Since $p(x,\nu(x))$ is invertible, uniqueness for the 
ordinary differential equation gives $w\equiv0$.
\end{proof}


\subsection{Spatial recovery and decay}

The elliptic boundary estimate first gives an equivalence between time derivatives and full spatial regularity. We then use
this equivalence to close the nonlinear energy estimate.
For a solution $U\in\mathcal X^s([0,T])$, set
\begin{equation*}
 h(t):=\norm{U(t,\cdot)}_{H^s(\Om)},\qquad 0\le t\le T.
\end{equation*}
Fix $0<R\le1$ so that $h(t)\le R$ keeps $U$ in a neighborhood of $0$.
The constants $C_R$ in the nonlinear estimates below are fixed
since $R$ is fixed. 

\begin{lemma}[Time derivatives and spatial regularity]
\label{lem:time-space-equivalence}
Under the domain and coefficient hypotheses of
Theorem~\ref{thm:smooth}, there exist $r_*>0$ and $C\ge1$ such
that every solution $U\in\mathcal X^s([0,T])$ of \eqref{eq:system}
with $\sup_{0\le t\le T}\norm{U(t,\cdot)}_{H^s(\Om)}\le r_*$ satisfies
\begin{equation}\label{eq:time-equivalence}
 C^{-1}\norm{U(t,\cdot)}_{H^s(\Om)}^2
 \le\sum_{k=0}^s\norm{\pa_t^kU(t,\cdot)}_{L^2(\Om)}^2
 \le C\norm{U(t,\cdot)}_{H^s(\Om)}^2
 \qquad(0\le t\le T).
\end{equation}
The constant $C$ is independent of $U$ and $T$.
\end{lemma}

\begin{proof}
Proposition~\ref{prop:app-time} gives
\begin{equation*}
 \sum_{k=0}^s\norm{\pa_t^kU(t,\cdot)}_{H^{s-k}(\Om)}\le C_Rh(t).
\end{equation*}
For the other direction, we separate the linear feedback from
its nonlinear remainder:
\begin{equation*}
 \zeta_-=K(x)\zeta_++N(x,\zeta_+),
\end{equation*}
where $N(x,\zeta_+)=\Phi(x,\zeta_+)-K(x)\zeta_+$.
Since $K=K(x)$ is independent of time, differentiating the boundary
condition $k$ times in $t$ gives
\begin{equation}\label{eq:time-boundary}
 B_0\pa_t^kU=\pa_t^kN(x,\zeta_+)\qquad(0\le k\le s),
\end{equation}
where $B_0U:=L_-(x)U-K(x)L_+(x)U$ is the linear boundary operator. Moreover, Proposition~\ref{prop:app-boundary} gives
\begin{equation}\label{eq:time-boundary-errors}
 \norm{(\pa_t^kN(\cdot,\zeta_+))(t,\cdot)}_{H^{s-k-1/2}(\Ga)}
 \le C_Rh^2(t)\qquad(0\le k<s).
\end{equation}

To apply the elliptic estimate, we write the equation as
\begin{equation} \label{eq:U}
 U_t+\sum_{j=1}^dA_j(0)U_{x_j}=Q_0U+\mathcal R(U),
\end{equation}
where $\mathcal R(U)=Q(U)-Q_0U-\sum_{j=1}^d(A_j(U)-A_j(0))U_{x_j}$.
Recalling the definition of the elliptic operator $P_0v=\sum_{j=1}^dA_j(0)v_{x_j}-Q_0v$ and differentiating \eqref{eq:U} in time,
we obtain, for $0\le k<s$,
\begin{equation*}
 P_0\pa_t^kU=-\pa_t^{k+1}U+\pa_t^k\mathcal R(U),\qquad
 \norm{(\pa_t^k\mathcal R(U))(t,\cdot)}_{H^{s-k-1}(\Om)}\le C_Rh^2(t),
\end{equation*}
where the estimate follows from Proposition~\ref{prop:app-interior}.
Lemma~\ref{lem:elliptic-estimate}, with $m=s-k$, now gives
\begin{equation*}
 \begin{aligned}
 \norm{\pa_t^kU(t,\cdot)}_{H^{s-k}(\Om)}
 &\le C\bigl(\norm{P_0\pa_t^kU(t,\cdot)}_{H^{s-k-1}(\Om)}
       +\norm{B_0\pa_t^kU(t,\cdot)}_{H^{s-k-1/2}(\Ga)}
       +\norm{\pa_t^kU(t,\cdot)}_{L^2(\Om)}\bigr)\\
 &\le C\bigl(\norm{\pa_t^{k+1}U(t,\cdot)}_{H^{s-k-1}(\Om)}
       +\norm{\pa_t^kU(t,\cdot)}_{L^2(\Om)}\bigr)+C_Rh^2(t),
 \end{aligned}
\end{equation*}
using \eqref{eq:time-boundary}--\eqref{eq:time-boundary-errors} in the
second line. At $k=s-1$, this reads
\begin{equation*}
 \norm{\pa_t^{s-1}U(t,\cdot)}_{H^1(\Om)}
 \le C\bigl(\norm{\pa_t^sU(t,\cdot)}_{L^2(\Om)}
       +\norm{\pa_t^{s-1}U(t,\cdot)}_{L^2(\Om)}\bigr)+C_Rh^2(t).
\end{equation*}
Substituting this into the estimate for $k=s-2$ and continuing downwards yields
\begin{equation*}
 \norm{\pa_t^kU(t,\cdot)}_{H^{s-k}(\Om)}
 \le C\sum_{j=k}^s\norm{\pa_t^jU(t,\cdot)}_{L^2(\Om)}+C_Rh^2(t)
 \qquad(0\le k<s).
\end{equation*}
Taking $k=0$ and applying the Cauchy--Schwarz inequality to the finite sum gives
\begin{equation*}
 h(t)\le C\left(\sum_{j=0}^s\norm{\pa_t^jU(t,\cdot)}_{L^2(\Om)}^2\right)^{1/2}
       +C_Rh^2(t).
\end{equation*}
Choosing $r_*$ sufficiently small gives $C_Rh^2(t)\le\frac12h(t)$.
This term can therefore be absorbed into the left-hand side, proving
the other inequality in \eqref{eq:time-equivalence}.
\end{proof}

Now we are in a position to prove the main result.
\begin{proof}[Proof of Theorem~\ref{thm:smooth}]
Let $U$ be the maximal solution given by Theorem~\ref{thm:local},
and let $0<T<T_{\max}$ be such that $h(t)\le r_*\le R$ on $[0,T]$,
keeping $R$ fixed.
Set $q(U)=Q(U)-Q_0U$. Proposition~\ref{prop:app-interior} gives
\begin{equation}\label{eq:time-equation}
 \pa_t^{k+1}U+\sum_jA_j(U)\pa_{x_j}\pa_t^kU=Q_0\pa_t^kU+G_k,
 \qquad \sum_{k=0}^s\norm{G_k(t,\cdot)}_{L^2(\Om)}\le C_Rh^2(t).
\end{equation}
Here $G_0=q(U)$, while for $k\ge1$,
\begin{equation*}
 G_k=\pa_t^kq(U)-\sum_j\sum_{\ell=1}^k
 \binom{k}{\ell}(\pa_t^\ell A_j(U))\pa_{x_j}\pa_t^{k-\ell}U.
\end{equation*}
Define
\begin{equation*}
 \Es(U(t))=\sum_{k=0}^s\int_\Om
 (\pa_t^kU)^\top M(x,U)(\pa_t^kU)\dd x.
\end{equation*}
This energy is equivalent to $h^2(t)$ by
Lemma~\ref{lem:time-space-equivalence}.
For each $k$, differentiation, \eqref{eq:time-equation}, and integration
by parts using the symmetry in \eqref{eq:symmetrizer} give
\begin{equation}\label{eq:energy-identity}
 \begin{aligned}
 \frac{\dd}{\dd t}\int_\Om(\pa_t^kU)^\top M(\pa_t^kU)\dd x
 &=-\int_\Ga(\pa_t^kU)^\top MA_\nu(\pa_t^kU)\dd S\\
 &\quad+\int_\Om(\pa_t^kU)^\top\mathcal D(U)(\pa_t^kU)\dd x
 +2\int_\Om(\pa_t^kU)^\top MG_k\dd x,
 \end{aligned}
\end{equation}
where $\mathcal D(U)=\pa_t(M(x,U))+\sum_j\pa_{x_j}(M(x,U)A_j(U))
+MQ_0+Q_0^\top M$.
We estimate these three terms separately.

\noindent(i) For the interior quadratic term, write
\begin{equation*}
 \mathcal D(U)=\mathcal D(0)+(\mathcal D(U)-\mathcal D(0)),\qquad
 \mathcal D(0)=\sum_j(M_0A_j(0))_{x_j}+M_0Q_0+Q_0^\top M_0.
\end{equation*}
The difference is
\begin{equation*}
 \begin{aligned}
 \mathcal D(U)-\mathcal D(0)
 ={}&M_U(x,U)(U_t)
 +\sum_j\bigl(M_{x_j}(x,U)A_j(U)-M_{x_j}(x,0)A_j(0)\bigr)\\
 &+\sum_j(MA_j)_U(x,U)(U_{x_j})
 +(M(x,U)-M_0)Q_0+Q_0^\top(M(x,U)-M_0).
 \end{aligned}
\end{equation*}
Proposition~\ref{prop:app-time} and Lemma~\ref{lem:app-trace-interpolation}
therefore imply
\begin{equation*}
 \norm{(\mathcal D(U)-\mathcal D(0))(t,\cdot)}_{L^\infty(\Om)}
 \le C_R\bigl(\norm{U(t,\cdot)}_{L^\infty(\Om)}
 +\norm{\nabla U(t,\cdot)}_{L^\infty(\Om)}
 +\norm{U_t(t,\cdot)}_{L^\infty(\Om)}\bigr)\le C_Rh(t).
\end{equation*}
By \eqref{eq:interior}, $\mathcal D(U)\le-2\lambda_0M_0+C_Rh(t)\Id$.
Also $\norm{M(\cdot,U(t,\cdot))-M_0}_{L^\infty(\Om)}\le C_Rh(t)$.
Enlarging the fixed constant $C_R$ to include this difference gives
\begin{equation}\label{eq:interior-perturbation}
 \mathcal D(U)\le-2\lambda_0M(x,U)+C_Rh(t)\Id.
\end{equation}
Consequently,
\begin{equation*}
 \sum_{k=0}^s\int_\Om(\pa_t^kU)^\top\mathcal D(U)(\pa_t^kU)\dd x
 \le-2\lambda_0\Es(U(t))+C_Rh^3(t).
\end{equation*}
\noindent(ii) For the source remainders, the Cauchy--Schwarz inequality and \eqref{eq:time-equation} give
\begin{equation*}
 2\left|\sum_{k=0}^s\int_\Om(\pa_t^kU)^\top MG_k\dd x\right|
 \le C\left(\sum_{k=0}^s\norm{\pa_t^kU(t,\cdot)}_{L^2(\Om)}^2\right)^{1/2}
       \left(\sum_{k=0}^s\norm{G_k(t,\cdot)}_{L^2(\Om)}^2\right)^{1/2}
 \le C_Rh^3(t).
\end{equation*}

\noindent(iii) For the boundary term, apply Lemma~\ref{lem:inhom-flux} with
$g=\pa_t^kN(x,\zeta_+)$ for $0\le k\le s$ and sum over $k$:
\begin{equation*}
 \begin{aligned}
 -\sum_{k=0}^s\int_\Ga(\pa_t^kU)^\top MA_\nu(\pa_t^kU)\dd S
 &\le-c\sum_{k=0}^s\norm{\pa_t^k\zeta_+(t,\cdot)}_{L^2(\Ga)}^2
       +C\sum_{k=0}^s\norm{(\pa_t^kN(\cdot,\zeta_+))(t,\cdot)}_{L^2(\Ga)}^2\\
 &\le-c\sum_{k=0}^s\norm{\pa_t^k\zeta_+(t,\cdot)}_{L^2(\Ga)}^2
       +C_Rh^2(t)\norm{\pa_t^s\zeta_+(t,\cdot)}_{L^2(\Ga)}^2
       +C_Rh^4(t)\\
 &\le-\bigl(c-C_Rh^2(t)\bigr)
       \sum_{k=0}^s\norm{\pa_t^k\zeta_+(t,\cdot)}_{L^2(\Ga)}^2+C_Rh^4(t)\\
 &\le C_Rh^4(t).
 \end{aligned}
\end{equation*}
The second inequality follows from Proposition~\ref{prop:app-boundary}(i).
In the last step, we chose $C_R r_*^2\le c/2$ and discarded
the remaining dissipation.
Summing \eqref{eq:energy-identity} and using energy equivalence yields
\begin{equation*}
 \frac{\dd}{\dd t}\Es(U(t))
 \le-2\lambda_0\Es(U(t))+C_R\Es^{3/2}(U(t))+C_R\Es^2(U(t)).
\end{equation*}
Fix $0<\lambda<\lambda_0$. Reducing $r_*$ further makes the last
two terms at most $2(\lambda_0-\lambda)\Es(U(t))$, so
\begin{equation*}
 \Es(U(t))\le e^{-2\lambda t}\Es(U_0),\qquad
 t \in [0,T].
\end{equation*}

By Lemma~\ref{lem:time-space-equivalence} and the bounds on $M$,
$h(t)\le C_*e^{-\lambda t}h(0)$ on $[0,T]$, where $C_*\ge1$
is independent of $T$.
Choose $r_*\le\delta_0$ and $\varepsilon\le r_*/(2C_*)$.
If $h(0)\le\varepsilon$ and $T_1<T_{\max}$ were the first time
with $h(T_1)=r_*$, the preceding estimate $h(T_1)\le C_*e^{-\lambda T_1}h(0)$ would give
$h(T_1)\le r_*/2$, a contradiction.
Thus $h(t)\le r_*/2\le\delta_0$ on $[0,T_{\max})$.
Theorem~\ref{thm:local} gives $T_{\max}=\infty$, and
$h(t)\le C_*e^{-\lambda t}h(0)$ for all $t\ge0$.
\end{proof}

\begin{remark}\label{rem:time-regularization}
At $k=s$, \eqref{eq:time-equation} formally contains $\pa_t^{s+1}U$.
For $U\in\mathcal X^s([0,T])$, this derivative is understood in
distributions. The energy calculation
is justified by time and tangential regularization; see
\cite[Theorem~9.19 and its proof]{BGS2007}.
The same interpretation applies to the mixed derivatives in the next
section and will not be repeated.
\end{remark}

\section{Proof of Theorem~\ref{thm:cylinder}}\label{sec:cylinder-proof}

Once the principal operator is not elliptic, the spatial recovery
from  time derivatives in Section~\ref{sec:smooth-proof} is no longer
available. The cylinder geometry now
plays a role: both boundary faces have the fixed normal direction
$x_1$, and the remaining variables $x'$ are tangential and periodic.
Non-characteristicity allows us to control the full $H^s(\Om)$ norm
by time and tangential spatial derivatives.

Write $Z^\alpha=\pa_t^{\alpha_0}\pa_{x'}^{\alpha'}$, with $\alpha$ a multi-index in $(t,x').$ We denote the order of $\alpha$ by
$|\alpha|=\alpha_0+|\alpha'|$.
As in Section~\ref{sec:smooth-proof}, set
$h(t)=\norm{U(t,\cdot)}_{H^s(\Om)}$ and fix a positive constant $R \leq 1$ so that $h(t)\le R$.

\begin{lemma}[Time and tangential derivatives and spatial regularity]
\label{lem:cylinder-space-equivalence}
Under the hypotheses of Theorem~\ref{thm:cylinder}, there are positive constants
$r_*>0$ and $C\ge1$ such that every solution
$U\in\mathcal X^s([0,T])$ of \eqref{eq:system} with
$h(t)\le r_*$ satisfies
\begin{equation}\label{eq:cylinder-equivalence}
 C^{-1}h^2(t)\le\sum_{|\alpha|\le s}\norm{Z^\alpha U(t,\cdot)}_{L^2(\Om)}^2
 \le Ch^2(t).
\end{equation}
The constant $C$ is independent of $U$ and $T$.
\end{lemma}

\begin{proof}
Proposition~\ref{prop:app-time}, followed by tangential differentiation,
gives $\norm{Z^\alpha U(t,\cdot)}_{H^{s-|\alpha|}(\Om)}\le C_Rh(t)$
for $|\alpha|\le s$, which proves the upper bound.
To recover normal derivatives, write
\begin{equation*}
 A_1(0)U_{x_1}=-U_t-\sum_{j=2}^dA_j(0)U_{x_j}+Q_0U+\mathcal R(U),
\end{equation*}
where $\mathcal R$ is the remainder in Proposition~\ref{prop:app-interior}.
We prove by induction on $p$ that, whenever $p+|\alpha|\le s$,
\begin{equation}\label{eq:normal-induction}
 \norm{\pa_{x_1}^pZ^\alpha U(t,\cdot)}_{L^2(\Om)}
 \le C_R\left\{\left(\sum_{|\beta|\le s}
 \norm{Z^\beta U(t,\cdot)}_{L^2(\Om)}^2\right)^{1/2}+h^2(t)\right\}.
\end{equation}
For $p=0$, the left side is one term in the sum, so the estimate
holds without the remainder. Assume it holds at $p-1$ for all
admissible time and tangential derivatives. Applying
$\pa_{x_1}^{p-1}Z^\alpha$ to the equation gives
\begin{equation} \label{eq:ZU}
 A_1(0)\pa_{x_1}^pZ^\alpha U
 =-\pa_{x_1}^{p-1}Z^\alpha U_t
       -\sum_{j=2}^dA_j(0)\pa_{x_1}^{p-1}Z^\alpha U_{x_j}
 +Q_0\pa_{x_1}^{p-1}Z^\alpha U
       +\pa_{x_1}^{p-1}Z^\alpha\mathcal R(U).
\end{equation}
The first two terms contain $p-1$ normal derivatives and
$|\alpha|+1$ time or tangential derivatives. Their total order is
$p+|\alpha|\le s$, so the induction hypothesis applies. It also
applies to the lower-order $Q_0$ term. Since
$p-1+|\alpha|\le s-1$, Proposition~\ref{prop:app-interior} gives
\begin{equation*}
 \norm{(\pa_{x_1}^{p-1}Z^\alpha\mathcal R(U))(t,\cdot)}_{L^2(\Om)}
 \le C_Rh^2(t).
\end{equation*}
Multiplication of \eqref{eq:ZU} by $A_1^{-1}(0)$ therefore proves
\eqref{eq:normal-induction} at order $p$, after enlarging $C_R$.
After the finitely many steps $p\le s$, taking the index $\alpha_0$ of time derivatives to be $0$ covers
all spatial derivatives of total order at most $s$. Summing up all the inequalities gives
\begin{equation*}
 h(t)\le C\left(\sum_{|\alpha|\le s}\norm{Z^\alpha U(t,\cdot)}_{L^2(\Om)}^2\right)^{1/2}
       +C_Rh^2(t).
\end{equation*}
Choose $r_*\le R$ so that $C_R r_*\le1/2$. Then
\begin{equation*}
 \frac12h(t)\le h(t)-C_Rh^2(t)
 \le C\left(\sum_{|\alpha|\le s}\norm{Z^\alpha U(t,\cdot)}_{L^2(\Om)}^2\right)^{1/2}.
\end{equation*}
Squaring proves the lower bound in \eqref{eq:cylinder-equivalence}.
\end{proof}

\begin{proof}[Proof of Theorem~\ref{thm:cylinder}]
Fix $0<\lambda<\lambda_0$. Let $U$ be the maximal solution given by Theorem~\ref{thm:local},
and let $0<T<T_{\max}$ be such that $h(t)\le r_*\le R$ on $[0,T]$,
keeping $R$ fixed.
For $0\le m\le s$, define
\begin{equation*}
 \begin{aligned}
 E_m(U)&=\sum_{|\alpha|=m}\int_\Om
 (Z^\alpha U)^\top M(x,U)(Z^\alpha U)\dd x,\\
 D_m(U)&=\sum_{i=0}^1\sum_{|\alpha|=m}
 \norm{Z^\alpha\zeta_{i,+}(t,\cdot)}_{L^2(\Ga_i)}^2.
 \end{aligned}
\end{equation*}

For $|\alpha|\le s$, differentiating the equation gives
\begin{equation*}
 (Z^\alpha U)_t+\sum_jA_j(U)(Z^\alpha U)_{x_j}=Q_0Z^\alpha U+F_\alpha,
\end{equation*}
where
\begin{equation*}
 F_\alpha=Z^\alpha q(U)-\sum_j[Z^\alpha,A_j(U)]U_{x_j},
 \qquad q(U)=Q(U)-Q_0U.
\end{equation*}
The second estimate in Proposition~\ref{prop:app-interior}, applied
to $Z^\alpha$, gives
\begin{equation*}
 \sum_{|\alpha|\le s}\norm{F_\alpha(t,\cdot)}_{L^2(\Om)}\le C_Rh^2(t).
\end{equation*}
Differentiating $\int_\Om (Z^\alpha U)^\top M(Z^\alpha U)\dd x$ with respect
to time and integrating by parts, we obtain
\begin{equation*}
 \begin{aligned}
 \frac{\dd}{\dd t}\int_\Om (Z^\alpha U)^\top M(Z^\alpha U)\dd x
 &=-\sum_{i=0}^1\int_{\Ga_i}(Z^\alpha U)^\top MA_{\nu_i}(Z^\alpha U)\dd S\\
 &\quad+\int_\Om (Z^\alpha U)^\top\mathcal D(U)(Z^\alpha U)\dd x
 +2\int_\Om (Z^\alpha U)^\top MF_\alpha\dd x,
 \end{aligned}
\end{equation*}
since $\Ga=\{0,1\}\times\T^{d-1}=\Ga_0\cup\Ga_1$.
We first estimate the boundary term.
On $\Ga_i$, write the boundary graph as
\begin{equation*}
 \zeta_{i,-}=K_i(x')\zeta_{i,+}+N_i(x',\zeta_{i,+}),
\end{equation*}
where $N_i(x',\zeta_{i,+})=\Phi_i(x',\zeta_{i,+})-K_i(x')\zeta_{i,+}$.
For $1\le|\alpha|=m\le s$, differentiation gives
\begin{equation*}
 Z^\alpha\zeta_{i,-}=K_i(x')Z^\alpha\zeta_{i,+}
 +[Z^\alpha,K_i]\zeta_{i,+}+Z^\alpha N_i(x',\zeta_{i,+}).
\end{equation*}
Proposition~\ref{prop:app-boundary} gives
\begin{equation*}
 \sum_{i,|\alpha|=m}\norm{[Z^\alpha,K_i]\zeta_{i,+}(t,\cdot)}_{L^2(\Ga_i)}^2
 \le C_K\sum_{j<m}D_j,\quad
 \sum_{i,|\alpha|=m}\norm{Z^\alpha N_i(x',\zeta_{i,+})(t,\cdot)}_{L^2(\Ga_i)}^2
 \le C_Rh^2(t)\sum_{j=0}^sD_j.
\end{equation*}
The constant $C_K$ depends on the fixed $C^s$ bounds of $K_i$
and need not be small. The full nonlinear remainder is small
relative to the boundary derivative norms.
Applying Lemma~\ref{lem:inhom-flux} with
$g=[Z^\alpha,K_i]\zeta_{i,+}+Z^\alpha N_i(x',\zeta_{i,+})$ yields

\begin{equation*}
 -\sum_{i,|\alpha|=m}\int_{\Ga_i}
 (Z^\alpha U)^\top MA_{\nu_i}(Z^\alpha U)\dd S
 \le-cD_m+C_K\sum_{j<m}D_j+C_Rh^2(t)\sum_{j=0}^sD_j.
\end{equation*}

For the interior quadratic term, we directly use \eqref{eq:interior-perturbation} and obtain
\begin{equation*}
 \begin{aligned}
 \sum_{|\alpha|=m}\int_\Om (Z^\alpha U)^\top\mathcal D(U)(Z^\alpha U)\dd x
 &\le-2\lambda_0E_m(U(t))
       +C_Rh(t)\sum_{|\alpha|=m}\norm{Z^\alpha U}_{L^2(\Om)}^2\\
 &\le-2\lambda_0E_m(U(t))+C_Rh^3(t).
 \end{aligned}
\end{equation*}
For the source remainders, the Cauchy--Schwarz inequality gives
\begin{equation*}
 \begin{aligned}
 2\left|\sum_{|\alpha|=m}\int_\Om (Z^\alpha U)^\top MF_\alpha\dd x\right|
 &\le C\left(\sum_{|\alpha|=m}\norm{Z^\alpha U(t,\cdot)}_{L^2(\Om)}^2\right)^{1/2}
       \left(\sum_{|\alpha|=m}\norm{F_\alpha(t,\cdot)}_{L^2(\Om)}^2\right)^{1/2}\\
 &\le C_Rh^3(t).
 \end{aligned}
\end{equation*}
Combining these bounds with the boundary estimate yields, for $0\le m\le s$,
\begin{equation} \label{eq:m}
 E_m'+2\lambda_0E_m+cD_m
 \le C_K\sum_{j<m}D_j+C_Rh^2(t)\sum_{j=0}^sD_j+C_Rh^3(t).
\end{equation}
To absorb the lower-order boundary terms, choose $0<\theta<1$ and set
\begin{equation*}
 E(U)=\sum_{m=0}^s\theta^m E_m(U),\qquad
 D(U)=\sum_{m=0}^s\theta^m D_m(U).
\end{equation*}
For fixed $\theta$, uniform positivity of $M$ and
\eqref{eq:cylinder-equivalence} show that $E(U(t))$ is equivalent to $h^2(t)$. For each $0\le m\le s$, we multiply the energy inequality \eqref{eq:m}  by $\theta^m$; summing over $m$, we obtain
\begin{equation*}
 E'+2\lambda_0E+cD
 \le C_K\sum_{m=1}^s\theta^m\sum_{j<m}D_j+C_Rh^2(t)\sum_{m=0}^s\theta^m\sum_{j=0}^sD_j+\frac{C_R}{1-\theta}h^3(t).
\end{equation*}
Moreover, direct computation gives
\begin{equation*}
 \begin{aligned}
 \sum_{m=1}^s\theta^m\sum_{j<m}D_j
 &=\sum_{j=0}^{s-1}D_j\sum_{m=j+1}^s\theta^m
 \le\frac{\theta}{1-\theta}\sum_{j=0}^{s-1}\theta^jD_j
 \le\frac{\theta}{1-\theta}D,\\
 \sum_{m=0}^s\theta^m\sum_{j=0}^sD_j
 &=\frac{1-\theta^{s+1}}{1-\theta}\sum_{j=0}^sD_j
 \le\frac1{1-\theta}\sum_{j=0}^sD_j
 \le\frac{\theta^{-s}}{1-\theta}D.
 \end{aligned}
\end{equation*}
In the last inequality, we used the fact that $\theta^j \ge\theta^s $ for all
$0\le j\le s$, since $0<\theta<1$.
Finally, we have
\begin{equation*}
 E'+2\lambda_0E+cD
 \le\left(\frac{C_K\theta}{1-\theta}
 +\frac{C_R\theta^{-s}}{1-\theta}h^2(t)\right)D
 +\frac{C_R}{1-\theta}h^3(t).
\end{equation*}
First fix $0<\theta<1$ so that $C_K\theta/(1-\theta)\le c/4$.
Then choose $r_*$ sufficiently small that
$C_R\theta^{-s}r_*^2/(1-\theta)\le c/4$.
Since $h(t)\le r_*$, each of the two boundary terms on the
right is at most $cD/4$. Absorbing them leaves
$cD/2$ on the left.
For the interior remainder, energy equivalence gives
$h^2(t)\le C_\theta E(U(t))$, and hence
$C_Rh^3(t)/(1-\theta)\le C_{R,\theta}h(t)E(U(t))$.
We conclude that
\begin{equation}\label{eq:cylinder-weighted-energy}
 E'+2\lambda_0E+\frac c2D\le C_{R,\theta}h(t)E.
\end{equation}

Reduce $r_*$ once more so that
$C_{R,\theta}r_*\le2(\lambda_0-\lambda)$. By \eqref{eq:cylinder-weighted-energy}, we have
$E(U(t))\le e^{-2\lambda t}E(U_0)$.
By energy equivalence and the continuity and continuation argument
in the proof of Theorem~\ref{thm:smooth}, we obtain
$T_{\max}=\infty$ and
$h(t)\le C_*e^{-\lambda t}h(0)$ for all $t\ge0$.
\end{proof}

\section{Failure of a general \texorpdfstring{$C^1$}{C1} decay estimate}
\label{sec:c1-counterexample}

By Sobolev embedding, the stability results proved in the previous sections already imply
\begin{equation*}
 \norm{U(t,\cdot)}_{C^1(\overline\Om)}
 \le C e^{-\lambda t}\norm{U_0}_{H^s(\Om)},
 \qquad s>d/2+1.
\end{equation*}
The issue is whether or not the initial norm on the right can be replaced
by $\norm{U_0}_{C^1(\overline\Om)}.$ Under the assumptions in this paper, the replacement is not valid in general, even for linear systems. The following example shows that the $C^1$ norm of the solution can grow rapidly compared to the $C^1$ norm of the initial data,
even when the $H^s$ norm of the initial data is arbitrarily small. 

Let $\Om=B_3(0)=\{x\in\R^2:|x|<3\}$. Fix $\gamma>0$ and set
\begin{equation*}
 A_1=\begin{pmatrix}1&0\\0&-1\end{pmatrix},\qquad
 A_2=\begin{pmatrix}0&1\\1&0\end{pmatrix},\qquad
 A_\nu=\nu_1A_1+\nu_2A_2.
\end{equation*}
On a boundary chart write $\nu=(\cos\theta,\sin\theta)$. The incoming
and outgoing variables are
\begin{equation*}
 \zeta_-=-\sin(\theta/2)U_1+\cos(\theta/2)U_2,\qquad
 \zeta_+=\cos(\theta/2)U_1+\sin(\theta/2)U_2.
\end{equation*}
They correspond to the eigenvalues $-1$ and $1$ of $A_\nu$.
We consider the real-valued system
\begin{equation}\label{eq:c1-counterexample-system}
 \begin{cases}
 U_t+A_1U_{x_1}+A_2U_{x_2}=-\gamma U,
   &(t,x)\in(0,\infty)\times\Om,\\
 \zeta_-=0,
   &(t,x)\in(0,\infty)\times\Ga,\\
 U(0,x)=U_0(x),&x\in\Om.
 \end{cases}
\end{equation}
Both matrices are symmetric, and
$A_1^2=A_2^2=\Id$, $A_1A_2+A_2A_1=0$.
Hence $A_\nu^2=\Id$ and the boundary is non-characteristic.
The condition $\zeta_-=0$ is the zero graph $\Phi=0$ of
Section~\ref{sec:introduction}.
With $G=\R^2$ and $M=\Id$, the left-hand side of
\eqref{eq:interior} is $-2\gamma\Id$.
Moreover, $\zeta_-=0$ implies $A_\nu U=U$, so the outgoing flux is
$U^\top A_\nu U=|U|^2$.
Thus Assumptions~\ref{ass:sym} and \ref{ass:boundary} hold with
$\lambda_0=\gamma$ and $\delta=1$.

This system satisfies the hypotheses of Theorem~\ref{thm:smooth},
but its propagation need not be bounded in the $C^1$ norm.
\begin{proposition}\label{prop:c1-counterexample}
For every integer $s>2$, there are nonzero initial data
$\widetilde U_0^\ell\in C_c^\infty(\Om;\R^2)$, satisfying the
compatibility conditions of every order, whose solutions of
\eqref{eq:c1-counterexample-system} obey
\begin{equation*}
 \norm{\widetilde U_0^\ell}_{H^s(\Om)}\longrightarrow0,
 \qquad
 \frac{\norm{\widetilde U^\ell(1,\cdot)}_{C^1(\overline\Om)}}
      {\norm{\widetilde U_0^\ell}_{C^1(\overline\Om)}}
 \longrightarrow\infty
 \quad\text{as }\ell\downarrow0.
\end{equation*}
In particular, there are no constants $C,\lambda,\varepsilon>0$
such that
\begin{equation*}
 \norm{U(t,\cdot)}_{C^1(\overline\Om)}
 \le C e^{-\lambda t}\norm{U_0}_{C^1(\overline\Om)}
 \qquad(t\ge0)
\end{equation*}
for all smooth compatible initial data with $\norm{U_0}_{H^s(\Om)}\le\varepsilon$.
\end{proposition}

\begin{proof}
Choose a nonzero nonnegative $\phi\in C_c^\infty((1,2))$.
For $0<\ell<1/4$, define
\begin{equation*}
 f_\ell(x)=\ell^s\phi\!\left(\frac{1-|x|}{\ell}\right),
 \qquad
 \widetilde U_0^\ell=(f_\ell,0)^\top.
\end{equation*}
We denote the solution of \eqref{eq:c1-counterexample-system} with this initial data by $\widetilde U^\ell=(\widetilde U_1^\ell,\widetilde U_2^\ell)^\top$.

Note that $\widetilde U_0^\ell$ is supported in the annulus
$1-2\ell<|x|<1-\ell$,
so all  compatibility conditions hold for any order.
Each spatial differentiation
costs at most a factor $\ell^{-1}$, so derivatives of order
$k\le s$ are $O(\ell^{s-k})$. The largest of these bounds occurs
at $k=s$, where $\ell^{s-s}=1$.
Since the support has area $O(\ell)$, integrating the squared
derivatives gives an $O(\ell)$ bound for the squared $H^s$ norm.
Taking the square root, we obtain
\begin{equation}\label{eq:c1-initial-bound}
 \norm{\widetilde U_0^\ell}_{H^s(\Om)}\le C\ell^{1/2},
 \qquad
 \norm{\widetilde U_0^\ell}_{C^1(\overline\Om)}\le C\ell^{s-1}.
\end{equation}

Let $v=e^{\gamma t}\widetilde U_2^\ell$.
Direct computation gives
\begin{equation*}
 v_{tt}-\Delta v=0,\qquad v(0,x)=0,\qquad
 v_t(0,x)=-(f_\ell)_{x_2}(x).
\end{equation*}
The speed of propagation is one and the initial data of $v$ are supported in $B_{1-\ell}(0)$.
Thus the whole-space solution is supported in
$\overline B_{1-\ell+t}(0)\subset B_2(0)$ for $0\le t\le1$;
in particular, $\operatorname{supp}v(1,\cdot)
\subseteq\overline B_{2-\ell}(0)$.
Since $\Om=B_3(0)$, it vanishes near $\Ga$. Consequently, for $0\le t\le1$, the solution coincides on $\Om$ with the solution of the Cauchy problem in $\R^2$ with the same initial data. Hence we can use the Poisson formula \cite[Section~2.4.1(c), p.~74]{Evans2010} for the solution of the wave equation in $\R^2$ to compute $v$ and its derivatives at $(t,x)=(1,0)$:
\begin{align*}
 -v_{x_2}(1,0)
 =\frac1{2\pi}\int_{B_1(0)}
   \frac{(f_\ell)_{y_2y_2}(y)}{\sqrt{1-|y|^2}}\dd y
 =\frac{\ell^s}{2}\int_{1-2\ell}^{1-\ell}
   \frac{r(2+r^2)\phi((1-r)/\ell)}{(1-r^2)^{5/2}}\dd r,
\end{align*}
where we used integration by parts and
\begin{equation*}
 \partial_{y_2}^2(1-|y|^2)^{-1/2}
 =(1-|y|^2)^{-3/2}+3y_2^2(1-|y|^2)^{-5/2};
\end{equation*}
there are no boundary terms because $f_\ell$ is compactly supported
in $B_1(0)$. Substituting $r=1-\ell z$ and returning to $\widetilde U_2^\ell$, we obtain
\begin{equation}\label{eq:c1-amplification}
\begin{aligned}
 -(\widetilde U_2^\ell)_{x_2}(1,0)
 =\frac{e^{-\gamma}}2\ell^{s-3/2}
   \int_1^2
   \frac{(1-\ell z)(2+(1-\ell z)^2)\phi(z)}
        {(z(2-\ell z))^{5/2}}\dd z
 \ge c\ell^{s-3/2},
\end{aligned}
\end{equation}
where $c>0$ is independent of sufficiently small $\ell$.
Combining \eqref{eq:c1-initial-bound} and \eqref{eq:c1-amplification}
yields
\begin{equation*}
 \frac{\norm{\widetilde U^\ell(1,\cdot)}_{C^1(\overline\Om)}}
      {\norm{\widetilde U_0^\ell}_{C^1(\overline\Om)}}
 \ge c\ell^{-1/2}\longrightarrow\infty.
\end{equation*}
The ratio tends to infinity as $\ell\downarrow0$, which completes the proof.
\end{proof}

\begin{remark}
Although exponential $C^1$ decay can be obtained in 1-D
\cite{BC2016}, it fails in general in multi-D, essentially because
of  loss of regularity in wave propagation \cite{IRS2021}.
By contrast, in 1-D, d'Alembert's formula gives $C^1$ solutions from
$C^1$ initial data.
\end{remark}

\section{Conclusion}\label{sec:conclusion}
In this paper, we have proved global existence and exponential $H^s$ stability
near the zero equilibrium for multi-D quasilinear hyperbolic
systems with nonlinear boundary feedback, under the assumptions that the boundary is non-characteristic
and there are strict interior and boundary Lyapunov
inequalities for the linearized systems. The proof combines energy estimates with elliptic spatial
recovery on smooth bounded domains, or tangential estimates and
normal recovery on flat periodic cylinders. Finally, a counterexample shows
that these assumptions do not imply an exponential $C^1$ decay
estimate controlled by the initial $C^1$ norm.

Future work includes extending this stability principle to
non-constant steady states, where the linearized system has spatially
varying coefficients.  A further direction is to relax the strict interior Lyapunov
inequality using observability estimates \cite{BLR1992}
or operator-weighted Lyapunov functionals \cite{BaiHayat2026}.

\appendix
\section{Nonlinear estimates}\label{app:nonlinear}

We first recall standard tools from nonlinear analysis and then apply them
to prove the interior estimates in Propositions~\ref{prop:app-time}
and~\ref{prop:app-interior} and the boundary estimates in
Propositions~\ref{prop:app-boundary-derivatives} and~\ref{prop:app-boundary}.
These estimates are used in Sections~\ref{sec:smooth-proof}
and~\ref{sec:cylinder-proof}.
Throughout, $s>d/2+1$ is an integer, and $\Om$ is either the smooth
bounded domain of Theorem~\ref{thm:smooth} or the periodic cylinder
of Theorem~\ref{thm:cylinder}. $\Ga$ is the boundary of $\Om.$ For the cylinder, we also use the notation $\Ga=\Ga_0\cup\Ga_1$ for the union of the periodic faces $\{i\}\times\T^{d-1}, i=0,1$. For multi-indices $\alpha=(\alpha_1,\ldots,\alpha_m),\beta=(\beta_1,\ldots,\beta_m)\in\mathbb N_0^m$, we set
$|\alpha|=\sum\nolimits_j\alpha_j$ and write $\beta\le\alpha$ if
$\beta_j\le\alpha_j$ for every component $j$.

We recall three standard Sobolev estimates from classical texts.
We adapt them to the domains and boundaries considered here by standard
extension and localization where needed, and we omit the details.
By \cite[Theorems~4.12 and~7.39, Remark~7.45]{AdamsFournier2003} we obtain the
following embedding and trace estimates.
\begin{lemma}[Embedding and traces]
\label{lem:app-trace-interpolation}
Let $s>d/2+1$ be an integer. Then
\begin{equation*}
 \norm v_{W^{1,\infty}(\Om)}\le C\norm v_{H^s(\Om)},
 \qquad
 \norm w_{L^\infty(\Om)}\le C\norm w_{H^{s-1}(\Om)}.
\end{equation*}
For every integer $1\le r\le s$,
\begin{equation*}
 \norm{v|_{\Ga}}_{H^{r-1/2}(\Ga)}
 \le C\norm v_{H^r(\Om)}.
\end{equation*}
The constants are independent of $v$ and $w$.
\end{lemma}

The following multiplication estimate follows from
\cite[Theorem~8.3.1]{Hormander1997}. This estimate will be frequently used to control the nonlinear remainders.
\begin{lemma}[Multiplication]
\label{lem:app-multiplication}
Let $s>d/2+1$ be an integer, and set $(S,m)=(\Om,d)$ or
$(\Ga,d-1)$. If $a,b,r\ge0$, $r\le\min\{a,b\}$, and
$a+b-r>m/2$, then
\begin{equation*}
 \norm{fg}_{H^r(S)}\le C\norm f_{H^a(S)}\norm g_{H^b(S)}.
\end{equation*}
The boundary estimate also holds on each periodic face $S=\Ga_i$.
\end{lemma}

\begin{remark}\label{rem:app-multiplication-cases}
We use the following choices.
\begin{enumerate}[(i)]
\item Let $(S,m)=(\Om,d)$. For integers $r,\ell\ge0$ with $r+\ell\le s-1$,
take $a=s-1-\ell$ and $b=r+\ell$ to obtain
\begin{equation}\label{eq:app-product-interior}
 \norm{fg}_{H^r(\Om)}
 \le C\norm f_{H^{s-1-\ell}(\Om)}\norm g_{H^{r+\ell}(\Om)}.
\end{equation}
Here $a+b-r=s-1>d/2$.
\item Let $(S,m)=(\Ga,d-1)$. For integers $1\le k\le s-1$ and $0\le\ell\le k-1$,
take $r=s-k-1/2$, $a=s-\ell-1/2$, and $b=s-k+\ell-1/2$ to obtain
\begin{equation}\label{eq:app-product-boundary-lower}
 \norm{fg}_{H^{s-k-1/2}(\Ga)}
 \le C\norm f_{H^{s-\ell-1/2}(\Ga)}
       \norm g_{H^{s-k+\ell-1/2}(\Ga)}.
\end{equation}
Here $a,b\ge r$ and $a+b-r=s-1/2>(d-1)/2$.
\item Let $(S,m)=(\Ga,d-1)$. For integers $1\le\ell\le s-1$,
take $r=0$, $a=s-\ell-1/2$, and $b=\ell-1/2$ to obtain
\begin{equation}\label{eq:app-product-boundary-top}
 \norm{fg}_{L^2(\Ga)}
 \le C\norm f_{H^{s-\ell-1/2}(\Ga)}\norm g_{H^{\ell-1/2}(\Ga)}.
\end{equation}
Here $a,b\ge1/2$ and $a+b=s-1>(d-1)/2$.
\end{enumerate}
\end{remark}

The composition estimates below control coefficient perturbations and
quadratic remainders. They are standard consequences of
\cite[Corollary~6.4.5 and Theorem~8.5.1]{Hormander1997}.
\begin{lemma}[Composition estimates]\label{lem:app-moser}
Let $s>d/2+1$ be an integer. Set $r=s$ for $S=\Om$, and
$r=s-1/2$ for $S=\Ga$ or a periodic face $S=\Ga_i$.
Let $F(x,u)$ be smooth with $F(x,0)=0$.
If $u\in H^r(S)$ satisfies
$\norm u_{L^\infty(S)}\le R$, then
\begin{equation*}
 \norm{F(\cdot,u(\cdot))}_{H^r(S)}\le C_R\norm u_{H^r(S)}.
\end{equation*}
If also $F_u(x,0)=0$, then
\begin{equation*}
 \norm{F(\cdot,u(\cdot))}_{H^r(S)}
 \le C_R\norm u_{L^\infty(S)}\norm u_{H^r(S)}.
\end{equation*}
Here $C_R>0$ depends on $R$ and is independent of $u$.
\end{lemma}

We now apply these tools to the equation. The next estimate controls
the upper bounds in the norm equivalences of
Lemmas~\ref{lem:time-space-equivalence} and~\ref{lem:cylinder-space-equivalence}.
\begin{proposition}[Estimates for time derivatives]\label{prop:app-time}
Let $U\in\mathcal X^s([0,T])$ solve the equation in \eqref{eq:system},
with $Q(0)=0$, and take values in a fixed closed ball contained in $G$.
At any fixed time, put $h(t)=\norm{U(t,\cdot)}_{H^s(\Om)}\le R$. Then
\begin{equation}\label{eq:app-time-derivatives}
 \norm{\pa_t^kU(t,\cdot)}_{H^{s-k}(\Om)}\le C_Rh(t)
 \qquad(0\le k\le s).
\end{equation}
Here $C_R>0$ is a constant depending on $R$ and independent of $U$.
\end{proposition}

\begin{proof}
For each fixed smooth coefficient function $F$, we also establish
\begin{equation}\label{eq:app-coefficient-time}
 \norm{(\pa_t^k(F(U)-F(0)))(t,\cdot)}_{H^{s-k}(\Om)}\le C_Rh(t).
\end{equation}
We prove \eqref{eq:app-time-derivatives} and
\eqref{eq:app-coefficient-time} simultaneously by induction on $k$.
At $k=0$, the solution bound is the definition of $h(t)$, and
\eqref{eq:app-coefficient-time} follows from
Lemma~\ref{lem:app-moser}.

Suppose both bounds hold up to order $k-1$, where $1\le k\le s$.
Applying $\pa_t^{k-1}$ to the equation in \eqref{eq:system} gives
\begin{equation*}
 \pa_t^kU=\pa_t^{k-1}Q(U)
 -\sum_j\sum_{\ell=0}^{k-1}\binom{k-1}{\ell}
   (\pa_t^\ell A_j(U))\pa_{x_j}\pa_t^{k-1-\ell}U.
\end{equation*}
The induction estimate (\ref{eq:app-coefficient-time}) and $Q(0)=0$ bound $\pa_t^{k-1}Q(U)$ by
$C_Rh(t)$ in $H^{s-k}(\Om)$. Moreover, in each product, for $0\le\ell\le k-1$, by the induction hypotheses \eqref{eq:app-time-derivatives} and
\eqref{eq:app-coefficient-time}, we have
\begin{equation*}
 \norm{(\pa_t^\ell A_j(U))(t,\cdot)}_{H^{s-1-\ell}(\Om)}\le C_R,
 \qquad
 \norm{\pa_{x_j}\pa_t^{k-1-\ell}U(t,\cdot)}_{H^{s-k+\ell}(\Om)}\le C_Rh(t).
\end{equation*}
Note that, for $\ell=0$, we use
$A_j(U)=A_j(0)+(A_j(U)-A_j(0))$ and absorb
$\norm{A_j(0)}_{H^{s-1}(\Om)}$ into $C_R$. Finally, we use estimate~\eqref{eq:app-product-interior}, with $r=s-k$ and
$r+\ell\le s-1$, to obtain $\norm{\pa_t^kU(t,\cdot)}_{H^{s-k}(\Om)}\le C_Rh(t)$.

It remains to prove \eqref{eq:app-coefficient-time} at order $k$.
The chain rule gives
\begin{equation*}
 \pa_t^kF(U)=\sum_{\ell=0}^{k-1}\binom{k-1}{\ell}
   (\pa_t^\ell F_U(U))\pa_t^{k-\ell}U.
\end{equation*}
As above, $\pa_t^\ell F_U(U)$ has $H^{s-1-\ell}(\Om)$ norm at most $C_R$,
and $\pa_t^{k-\ell}U$ has $H^{s-k+\ell}(\Om)$ norm at most $C_Rh(t)$.
When $\ell=0$, we use the bound for $\pa_t^kU$ just obtained.
The same multiplication estimate proves
\eqref{eq:app-coefficient-time} and closes the induction.
\end{proof}

The next proposition gives quadratic bounds for the nonlinear remainders.
\begin{proposition}[Interior remainder estimates]\label{prop:app-interior}
Let $U$ be as in Proposition~\ref{prop:app-time}, with $h(t)\le R\le1$.
Write $\partial^\beta=\pa_t^{\beta_0}\pa_x^{\beta'}$ and set
\begin{equation*}
 a_j(U)=A_j(U)-A_j(0),\qquad q(U)=Q(U)-Q_U(0)U,\qquad
 \mathcal R(U)=q(U)-\sum_j a_j(U)U_{x_j}.
\end{equation*}
Then
\begin{align*}
 &\sum_{|\beta|\le s-1}\norm{(\partial^\beta\mathcal R(U))(t,\cdot)}_{L^2(\Om)}\le C_Rh^2(t),\\
 &\sum_{|\beta|\le s}\left(
 \norm{(\partial^\beta q(U))(t,\cdot)}_{L^2(\Om)}
 +\sum_j\norm{([\partial^\beta,A_j(U)]U_{x_j})(t,\cdot)}_{L^2(\Om)}
 \right)\le C_Rh^2(t),
\end{align*}
where, for any smooth vector-valued function $v$, the commutator is
\begin{equation*}
 [\partial^\beta,A_j(U)]v
 :=\partial^\beta(A_j(U)v)-A_j(U)\partial^\beta v.
\end{equation*}
\end{proposition}

\begin{proof}
Proposition~\ref{prop:app-time} and
\eqref{eq:app-coefficient-time}
give
\begin{equation*}
 \norm{\partial^\beta U(t,\cdot)}_{H^{s-|\beta|}(\Om)}
 +\norm{(\partial^\beta a_j(U))(t,\cdot)}_{H^{s-|\beta|}(\Om)}\le C_Rh(t)
 \qquad(|\beta|\le s).
\end{equation*}
For $|\alpha|\le s$, expand
\begin{equation*}
 [\partial^\alpha,A_j(U)]U_{x_j}
 =\sum_{0<\beta\le\alpha}\binom\alpha\beta
   (\partial^\beta A_j(U))\partial^{\alpha-\beta}U_{x_j}.
\end{equation*}
Note that $\partial^\beta A_j(U)=\partial^\beta a_j(U)$ for $|\beta|\ge1$;
then $\partial^\beta A_j(U)$ has
$H^{s-|\beta|}(\Om)$ norm at most $C_Rh(t)$. Moreover, 
$\partial^{\alpha-\beta}U_{x_j}$ has $H^{|\beta|-1}(\Om)$ norm
at most $C_Rh(t)$.
Applying estimate~\eqref{eq:app-product-interior}, with $r=0$ and $\ell=|\beta|-1$,
to each term and summing over $\beta$ gives
\begin{equation*}
 \norm{[\partial^\alpha,A_j(U)]U_{x_j}}_{L^2(\Om)}\le C\sum_{0<\beta\le\alpha}
   \norm{\partial^\beta A_j(U)}_{H^{s-|\beta|}(\Om)}
   \norm{\partial^{\alpha-\beta}U_{x_j}}_{H^{|\beta|-1}(\Om)}\le C_Rh^2(t).
\end{equation*}
Similarly, for $|\alpha|\le s-1$, we have
\begin{equation*}
 \norm{\partial^\alpha(a_j(U)U_{x_j})}_{L^2(\Om)}\le C\sum_{\beta\le\alpha}
   \norm{\partial^\beta a_j(U)}_{H^{s-1-|\beta|}(\Om)}
   \norm{\partial^{\alpha-\beta}U_{x_j}}_{H^{|\beta|}(\Om)}\le C_Rh^2(t).
\end{equation*}

For $q$, note that $q_U(0)=0$. Lemma~\ref{lem:app-moser} and
Lemma~\ref{lem:app-trace-interpolation} give $\norm{q(U(t,\cdot))}_{H^s(\Om)}\le C_Rh^2(t)$.
For $1\le k\le s$, we have
\begin{equation*}
 \pa_t^kq(U)=\sum_{\ell=0}^{k-1}\binom{k-1}{\ell}
    (\pa_t^\ell q_U(U))\pa_t^{k-\ell}U.
\end{equation*}
For $0 \leq \ell \leq k-1$,
$\pa_t^\ell q_U(U)$ has $H^{s-1-\ell}(\Om)$ norm at most $C_Rh(t)$, and $\pa_t^{k-\ell}U$ has
$H^{s-k+\ell}(\Om)$ norm at most $C_Rh(t)$.
Estimate~\eqref{eq:app-product-interior}, with $r=s-k$, yields
\begin{equation*}
 \norm{(\pa_t^kq(U))(t,\cdot)}_{H^{s-k}(\Om)}\le C\sum_{\ell=0}^{k-1}
   \norm{\pa_t^\ell q_U(U)}_{H^{s-1-\ell}(\Om)}
   \norm{\pa_t^{k-\ell}U}_{H^{s-k+\ell}(\Om)}\le C_Rh^2(t),
\end{equation*}
which concludes the proof.
\end{proof}

To estimate the nonlinear boundary remainder in
Proposition~\ref{prop:app-boundary}, we first establish derivative
bounds for $F(x,\zeta_+)$ with smooth $F$.
\begin{proposition}[Estimates for boundary derivatives]
\label{prop:app-boundary-derivatives}
Let $U$ be as in Proposition~\ref{prop:app-time}, with
$h(t)=\norm{U(t,\cdot)}_{H^s(\Om)}\le R$.
\begin{enumerate}[(i)]
\item On the smooth boundary $\Ga$, recall that
$\zeta_+=\zeta_+(t,x)$ is the  outgoing variable at the boundary point $x \in \Gamma.$ For a smooth $F(x,z)$ with $F(x,0)=0$, we have
\begin{equation}\label{eq:app-boundary-time-composition}
 \norm{(\pa_t^k F(\cdot,\zeta_+))(t,\cdot)}_{H^{s-k-1/2}(\Ga)}
 \le C_Rh(t)\qquad(0\le k\le s-1).
\end{equation}
\item On each boundary  $\Ga_i$ of $\Om=(0,1)\times \mathbb T^{d-1}$, recall that
$\zeta_+=\zeta_+(t,x)$ is the  outgoing variable at the boundary point $x \in \Gamma_i.$ Note that $x=(0,x') \in \Ga_0$ or $x=(1,x') \in \Ga_1$. We denote the tangential derivatives by
$Z^\alpha=\pa_t^{\alpha_0}\pa_{x'}^{\alpha'}$, where $\alpha=(\alpha_0,\alpha')\in\mathbb N_0^d$ is a multi-index.
For a smooth $F(x',z)$ with $F(x',0)=0$, we have
\begin{equation}\label{eq:app-boundary-mixed-composition}
 \norm{(Z^\alpha F(\cdot,\zeta_{i,+}))(t,\cdot)}_{H^{s-|\alpha|-1/2}(\Ga_i)}
 \le C_Rh(t)\qquad(|\alpha|\le s-1).
\end{equation}
\end{enumerate}
Here $C_R>0$ depends on $R$ and is independent of $U$.
\end{proposition}

\begin{proof}
(i) Proposition~\ref{prop:app-time} and
Lemma~\ref{lem:app-trace-interpolation} give
\begin{equation*}
 \norm{\pa_t^k\zeta_+(t,\cdot)}_{H^{s-k-1/2}(\Ga)}\le C_Rh(t)
 \qquad(0\le k\le s-1).
\end{equation*}
We prove \eqref{eq:app-boundary-time-composition} by induction on $k$
for every such $F$. 
At $k=0$, we use Lemma~\ref{lem:app-moser}.
Assume the bound holds up to order $k-1$, where $1\le k\le s-1$.
The chain rule gives
\begin{equation*}
 \pa_t^k F(x,\zeta_+)
 =\sum_{\ell=0}^{k-1}\binom{k-1}{\ell}
   (\pa_t^\ell F_z(x,\zeta_+))\pa_t^{k-\ell}\zeta_+.
\end{equation*}
By induction, for $0\le\ell\le k-1$, $\pa_t^\ell F_z(x,\zeta_+)$
is bounded by $C_R$ in $H^{s-\ell-1/2}(\Ga)$.
The trace estimate above bounds $\pa_t^{k-\ell}\zeta_+$ by
$C_Rh(t)$ in $H^{s-k+\ell-1/2}(\Ga)$.
Estimate~\eqref{eq:app-product-boundary-lower} with $r=s-k-1/2$ gives
\begin{equation*}
 \begin{aligned}
 \norm{(\pa_t^kF(\cdot,\zeta_+))(t,\cdot)}_{H^{s-k-1/2}(\Ga)}
 \le C\sum_{\ell=0}^{k-1}
   \norm{\pa_t^\ell F_z(\cdot,\zeta_+)}_{H^{s-\ell-1/2}(\Ga)}
   \norm{\pa_t^{k-\ell}\zeta_+}_{H^{s-k+\ell-1/2}(\Ga)}\le C_Rh(t),
 \end{aligned}
\end{equation*}
which completes the induction.

(ii) Proposition~\ref{prop:app-time}, tangential differentiation, and
Lemma~\ref{lem:app-trace-interpolation} give
\begin{equation*}
 \norm{Z^\alpha\zeta_{i,+}(t,\cdot)}_{H^{s-|\alpha|-1/2}(\Ga_i)}\le C_Rh(t)
 \qquad(|\alpha|\le s-1).
\end{equation*}
Again we induct on $k=|\alpha|$ for every $F$ with $F(x',0)=0$.
The case $k=0$ follows from Lemma~\ref{lem:app-moser}.
For $k\ge1$, choose $j$ with $\alpha_j>0$ and put
$\widehat\alpha=\alpha-e_j$, where $e_j\in\mathbb{N}_0^d$ has $1$ in its $j$-th component and $0$ elsewhere. One application of the chain rule followed by Leibniz's rule yields
\begin{equation}\label{eq:app-boundary-chain}
 Z^\alpha F(x',\zeta_{i,+})
 =Z^{\widehat\alpha}F_{x_{j+1}}(x',\zeta_{i,+})
 +\sum_{\beta\le\widehat\alpha}\binom{\widehat\alpha}{\beta}
   (Z^\beta F_z(x',\zeta_{i,+}))Z^{\alpha-\beta}\zeta_{i,+}.
\end{equation}
Since $F_{x_{j+1}}(x',0)=0$, that term is controlled by the induction
hypothesis in $H^{s-k+1/2}(\Ga_i)$.
For $\ell=|\beta|\le k-1$, separate the fixed coefficient
$F_z(x',0)$ and apply the induction hypothesis to its perturbation.
The two product factors are bounded by $C_R$ and $C_Rh(t)$ in
$H^{s-\ell-1/2}(\Ga_i)$ and $H^{s-k+\ell-1/2}(\Ga_i)$, respectively.
Estimate~\eqref{eq:app-product-boundary-lower} with $r=s-k-1/2$ proves the claim.
\end{proof}

The estimates of boundary remainders below are used in spatial recovery in
Section~\ref{sec:smooth-proof} and boundary self-absorption in
Section~\ref{sec:cylinder-proof}.
\begin{proposition}[Boundary remainder estimates]\label{prop:app-boundary}
Let $U$ be as in Proposition~\ref{prop:app-time}, with $h(t)\le R\le1$.
For the boundary graph, set
\begin{equation*}
 K(x)=\Phi_{\zeta_+}(x,0),\qquad
 N(x,\zeta_+)=\Phi(x,\zeta_+)-K(x)\zeta_+.
\end{equation*}
Thus $N(x,0)=N_{\zeta_+}(x,0)=0$.
Highest-order boundary norms are interpreted as in
Remark~\ref{rem:time-regularization}.
\begin{enumerate}[(i)]
\item On the smooth boundary, let
$\zeta_+=L_+(x)U|_{\Ga}$. Then
\begin{equation*}
 \begin{aligned}
 \norm{(\pa_t^kN(\cdot,\zeta_+))(t,\cdot)}_{H^{s-k-1/2}(\Ga)}
 &\le C_Rh^2(t)\qquad(0\le k<s),\\
 \norm{(\pa_t^sN(\cdot,\zeta_+))(t,\cdot)}_{L^2(\Ga)}
 &\le C_Rh(t)\norm{\pa_t^s\zeta_+(t,\cdot)}_{L^2(\Ga)}+C_Rh^2(t).
 \end{aligned}
\end{equation*}
\item On a periodic face $\Ga_i$, let
$\zeta_+=L_+U|_{\Ga_i}$.
Then
\begin{align*}
 \sum_{|\alpha|\le s}\norm{(Z^\alpha N(\cdot,\zeta_+))(t,\cdot)}_{L^2(\Ga_i)}^2
 &\le C_Rh^2(t)\sum_{|\beta|\le s}\norm{Z^\beta\zeta_+(t,\cdot)}_{L^2(\Ga_i)}^2,\\
 \sum_{|\alpha|=m}\norm{[Z^\alpha,K]\zeta_+(t,\cdot)}_{L^2(\Ga_i)}^2
 &\le C_R\sum_{|\beta|\le m-1}\norm{Z^\beta\zeta_+(t,\cdot)}_{L^2(\Ga_i)}^2
 \qquad(1\le m\le s).
\end{align*}
\end{enumerate}
The constant $C_R$ is independent of $U$.
\end{proposition}

\begin{proof}
(i) At $k=0$, Lemma~\ref{lem:app-moser} gives the quadratic bound.
For $k\ge1$,
\begin{equation}\label{eq:app-boundary-N-time}
 \pa_t^kN(x,\zeta_+)
 =\sum_{\ell=0}^{k-1}\binom{k-1}{\ell}
   (\pa_t^\ell N_{\zeta_+}(x,\zeta_+))\pa_t^{k-\ell}\zeta_+.
\end{equation}
Because $N_{\zeta_+}(x,0)=0$, apply
\eqref{eq:app-boundary-time-composition} to $F=N_{\zeta_+}$.
The first factor has $H^{s-\ell-1/2}(\Ga)$ norm at most
$C_Rh(t)$. For $k<s$, the trace bound for the second factor and
\eqref{eq:app-product-boundary-lower} give the desired $C_Rh^2(t)$.
For $k=s$, the $\ell=0$ term is bounded by
\begin{equation*}
 \norm{N_{\zeta_+}(\cdot,\zeta_+(t,\cdot))}_{L^\infty(\Ga)}
 \norm{\pa_t^s\zeta_+(t,\cdot)}_{L^2(\Ga)}
 \le C_Rh(t)\norm{\pa_t^s\zeta_+(t,\cdot)}_{L^2(\Ga)}.
\end{equation*}
For $1\le\ell\le s-1$, the factors in
\eqref{eq:app-boundary-N-time} have $H^{s-\ell-1/2}(\Ga)$
and $H^{\ell-1/2}(\Ga)$ norms at most $C_Rh(t)$.
Estimate~\eqref{eq:app-product-boundary-top} gives $C_Rh^2(t)$ in
$L^2(\Ga)$.

(ii) For $i=0,1$, we have
\begin{equation}\label{eq:app-boundary-dissipation-input}
 \begin{aligned}
 \norm{Z^\beta \zeta_+(t,\cdot)}_{H^{s-k-1/2}(\Ga_i)}
 \le C\norm{Z^\beta \zeta_+(t,\cdot)}_{H^{s-k}(\Ga_i)}
 \le C\sum_{|\gamma|\le s}\norm{Z^\gamma\zeta_+(t,\cdot)}_{L^2(\Ga_i)},
 \qquad k=|\beta|\le s-1.
 \end{aligned}
\end{equation}
We prove by induction, for every smooth quadratic remainder $N$, that
\begin{equation}\label{eq:app-boundary-quadratic-induction}
 \norm{(Z^\alpha N(\cdot,\zeta_+))(t,\cdot)}_{H^{s-k-1/2}(\Ga_i)}
 \le C_Rh(t)\sum_{|\gamma|\le s}\norm{Z^\gamma\zeta_+(t,\cdot)}_{L^2(\Ga_i)},
 \qquad k=|\alpha|\le s-1.
\end{equation}
For $k=0$, Lemma~\ref{lem:app-moser} gives
\begin{equation*}
 \norm{(N(\cdot,\zeta_+))(t,\cdot)}_{H^{s-1/2}(\Ga_i)}
 \le C_R\norm{\zeta_+(t,\cdot)}_{L^\infty(\Ga_i)}\norm{\zeta_+(t,\cdot)}_{H^{s-1/2}(\Ga_i)}
 \le C_Rh(t)\sum_{|\gamma|\le s}\norm{Z^\gamma\zeta_+(t,\cdot)}_{L^2(\Ga_i)}.
\end{equation*}
For $1\le k\le s-1$, apply \eqref{eq:app-boundary-chain} with $F=N$.
The explicit derivative $N_{x_{j+1}}$ still vanishes quadratically, so
its term is controlled by the induction hypothesis at order $k-1$.
For each product, \eqref{eq:app-boundary-mixed-composition} applied
to $N_{\zeta_+}$ bounds its first factor by $C_Rh(t)$ in
$H^{s-\ell-1/2}(\Ga_i)$, where $\ell=|\beta|$.
Use \eqref{eq:app-boundary-dissipation-input} for the other factor in
$H^{s-k+\ell-1/2}(\Ga_i)$. Estimate~\eqref{eq:app-product-boundary-lower}
then proves \eqref{eq:app-boundary-quadratic-induction}.

At $k=s$, the explicit-position term is controlled by
\eqref{eq:app-boundary-quadratic-induction} at order $s-1$ in
$H^{1/2}(\Ga_i)$, hence in $L^2(\Ga_i)$.
The $\beta=0$ product satisfies
\begin{equation*}
 \norm{(N_{\zeta_+}\,Z^\alpha \zeta_+)(t,\cdot)}_{L^2(\Ga_i)}
 \le C_Rh(t)\norm{Z^\alpha \zeta_+(t,\cdot)}_{L^2(\Ga_i)}
 \le C_Rh(t)\sum_{|\gamma|\le s}\norm{Z^\gamma\zeta_+(t,\cdot)}_{L^2(\Ga_i)}.
\end{equation*}
For $1\le\ell=|\beta|\le s-1$, use
\eqref{eq:app-product-boundary-top}: the first factor is bounded by
$C_Rh(t)$ in $H^{s-\ell-1/2}(\Ga_i)$ and the second by
$C\sum_{|\gamma|\le s}\norm{Z^\gamma\zeta_+(t,\cdot)}_{L^2(\Ga_i)}$
in $H^{\ell-1/2}(\Ga_i)$.
Together with \eqref{eq:app-boundary-quadratic-induction}, squaring and
summing these bounds and applying Cauchy--Schwarz to the finite sums
gives the first estimate.
Finally, each term of $[Z^\alpha,K]\zeta_+$ differentiates $K$
at least once in $x'$, leaving a solution derivative of order less
than $|\alpha|$. The fixed smooth coefficient is bounded in
$L^\infty(\Ga_i)$, and summing the squared bounds gives the second estimate.
\end{proof}


\begin{thebibliography}{10}

\bibitem{AdamsFournier2003}
R.~A. Adams and J.~J.~F. Fournier.
\newblock {\em Sobolev Spaces}.
\newblock Academic Press, 2003.

\bibitem{ADN1964}
S.~Agmon, A.~Douglis, and L.~Nirenberg.
\newblock Estimates near the boundary for solutions of elliptic partial
  differential equations satisfying general boundary conditions. {II}.
\newblock {\em Communications on Pure and Applied Mathematics}, 1964.

\bibitem{BaiHayat2026}
Y.~Bai and A.~Hayat.
\newblock Lyapunov functions and stability of linear two-dimensional hyperbolic
  systems.
\newblock {\em HAL preprint}, 2026.

\bibitem{BallZarnescu2017}
J.~M. Ball and A.~Zarnescu.
\newblock Partial regularity and smooth topology-preserving approximations of
  rough domains.
\newblock {\em Calculus of Variations and Partial Differential Equations},
  2017.

\bibitem{BLR1992}
C.~Bardos, G.~Lebeau, and J.~Rauch.
\newblock Sharp sufficient conditions for the observation, control, and
  stabilization of waves from the boundary.
\newblock {\em SIAM Journal on Control and Optimization}, 1992.

\bibitem{BC2016}
G.~Bastin and J.-M. Coron.
\newblock {\em Stability and Boundary Stabilization of {1-D} Hyperbolic
  Systems}.
\newblock Birkh{\"a}user, 2016.

\bibitem{BGS2007}
S.~Benzoni-Gavage and D.~Serre.
\newblock {\em Multi-dimensional Hyperbolic Partial Differential Equations:
  First-order Systems and Applications}.
\newblock Oxford University Press, 2007.

\bibitem{CP1982}
J.~Chazarain and A.~Piriou.
\newblock {\em Introduction to the Theory of Linear Partial Differential
  Equations}.
\newblock North-Holland, 1982.

\bibitem{CB2015}
J.-M. Coron and G.~Bastin.
\newblock Dissipative boundary conditions for one-dimensional quasi-linear
  hyperbolic systems: {Lyapunov} stability for the {$C^1$}-norm.
\newblock {\em SIAM Journal on Control and Optimization}, 2015.

\bibitem{CBDN2008}
J.-M. Coron, G.~Bastin, and B.~{d'Andr{\'e}a-Novel}.
\newblock Dissipative boundary conditions for one-dimensional nonlinear
  hyperbolic systems.
\newblock {\em SIAM Journal on Control and Optimization}, 2008.

\bibitem{CBDN2007}
J.-M. Coron, B.~{d'Andr{\'e}a-Novel}, and G.~Bastin.
\newblock A strict {Lyapunov} function for boundary control of hyperbolic
  systems of conservation laws.
\newblock {\em IEEE Transactions on Automatic Control}, 2007.

\bibitem{CN2015}
J.-M. Coron and H.-M. Nguyen.
\newblock Dissipative boundary conditions for nonlinear {1-D} hyperbolic
  systems: Sharp conditions through an approach via time-delay systems.
\newblock {\em SIAM Journal on Mathematical Analysis}, 2015.

\bibitem{CVKB2013}
J.-M. Coron, R.~Vazquez, M.~Krstic, and G.~Bastin.
\newblock Local exponential {$H^2$} stabilization of a {$2\times2$} quasilinear
  hyperbolic system using backstepping.
\newblock {\em SIAM Journal on Control and Optimization}, 2013.

\bibitem{lit-c2d5f55c2bba}
I.~Duran and A.~S. Morgans.
\newblock {On the reflection and transmission of circumferential waves through
  nozzles}.
\newblock {\em Journal of Fluid Mechanics}, 2015.

\bibitem{Evans2010}
L.~C. Evans.
\newblock {\em Partial Differential Equations}.
\newblock American Mathematical Society, 2010.

\bibitem{GLTW2012}
M.~Gugat, G.~Leugering, S.~Tamasoiu, and K.~Wang.
\newblock {$H^2$}-stabilization of the isothermal {Euler} equations: a
  {Lyapunov} function approach.
\newblock {\em Chinese Annals of Mathematics, Series B}, 2012.

\bibitem{HS2021}
A.~Hayat and P.~Shang.
\newblock Exponential stability of density-velocity systems with boundary
  conditions and source term for the {$H^2$} norm.
\newblock {\em Journal de Math\'ematiques Pures et Appliqu\'ees}, 2021.

\bibitem{HTControl2024}
M.~Herty and F.~Thein.
\newblock Boundary feedback control for hyperbolic systems.
\newblock {\em ESAIM: Control, Optimisation and Calculus of Variations}, 2024.

\bibitem{Hormander1997}
L.~H{\"o}rmander.
\newblock {\em Lectures on Nonlinear Hyperbolic Differential Equations}.
\newblock Springer, 1997.

\bibitem{HDVK2016}
L.~Hu, F.~Di~Meglio, R.~Vazquez, and M.~Krstic.
\newblock Control of homodirectional and general heterodirectional linear
  coupled hyperbolic {PDEs}.
\newblock {\em IEEE Transactions on Automatic Control}, 2016.

\bibitem{HVDK2019}
L.~Hu, R.~Vazquez, F.~Di~Meglio, and M.~Krstic.
\newblock Boundary exponential stabilization of 1-dimensional inhomogeneous
  quasi-linear hyperbolic systems.
\newblock {\em SIAM Journal on Control and Optimization}, 2019.

\bibitem{IRS2021}
A.~Israelsson, S.~Rodr{\'i}guez-L{\'o}pez, and W.~Staubach.
\newblock Local and global estimates for hyperbolic equations in
  {Besov--Lipschitz} and {Triebel--Lizorkin} spaces.
\newblock {\em Analysis \& PDE}, 2021.

\bibitem{Li2010}
T.~Li.
\newblock {\em Controllability and Observability for Quasilinear Hyperbolic
  Systems}.
\newblock American Institute of Mathematical Sciences and Higher Education
  Press, 2010.

\bibitem{PS2020}
M.~Pokojovy and R.~Schnaubelt.
\newblock Boundary stabilization of quasilinear {Maxwell} equations.
\newblock {\em Journal of Differential Equations}, 2020.

\bibitem{PWB2008}
C.~Prieur, J.~Winkin, and G.~Bastin.
\newblock Robust boundary control of systems of conservation laws.
\newblock {\em Mathematics of Control, Signals, and Systems}, 2008.

\bibitem{Schochet1986}
S.~Schochet.
\newblock The compressible {Euler} equations in a bounded domain: Existence of
  solutions and the incompressible limit.
\newblock {\em Communications in Mathematical Physics}, 1986.

\bibitem{THM2016}
H.-G. Tang, B.-S. He, and H.-B. Mou.
\newblock {P}- and {S}-wave energy flux density vectors.
\newblock {\em Geophysics}, 2016.

\bibitem{Trakhinin2006}
Y.~Trakhinin.
\newblock Dissipative symmetrizers of hyperbolic problems and their
  applications to shock waves and characteristic discontinuities.
\newblock {\em SIAM Journal on Mathematical Analysis}, 2006.

\bibitem{WWY2020}
K.~Wang, Z.~Wang, and W.~Yao.
\newblock Boundary feedback stabilization of quasilinear hyperbolic systems
  with partially dissipative structure.
\newblock {\em Systems \& Control Letters}, 2020.

\bibitem{YY2024}
H.~Yang and W.-A. Yong.
\newblock Feedback boundary control of multi-dimensional hyperbolic systems
  with relaxation.
\newblock {\em Automatica}, 2024.

\bibitem{YY2025}
H.~Yang and W.-A. Yong.
\newblock Boundary control of multidimensional discrete-velocity kinetic
  models.
\newblock {\em IEEE Transactions on Automatic Control}, 2025.

\end{thebibliography}
\end{document}